\documentclass[reqno]{amsart}

\usepackage{amsmath} \usepackage{amssymb}

\usepackage[reftex]{theoremref}

\usepackage{enumitem}

\usepackage{tikz} \usepackage{tikz-cd} \usepackage{tikz-3dplot}

\pgfqkeys{/tikz/commutative diagrams}{%
  path/.style={/tikz/arrows=cm cap-cm cap}%
}
\makeatletter \define@key{meshkeys}{u}{\def\myu{#1}}
\define@key{meshkeys}{v}{\def\myv{#1}}
\define@key{meshkeys}{r}{\def\myr{#1}}
\tikzdeclarecoordinatesystem{mesh}%
{%
  \setkeys{meshkeys}{#1}%
  \pgfpointadd{\pgfpointxy{\myr*(\myu+\myv)*cos(45)}%
    {\myr*(\myv-\myu)*sin(45)}}%
}

 \numberwithin{figure}{section}
\numberwithin{table}{section} \numberwithin{equation}{section}

 \newtheorem{theorem}{Theorem}[section]
\newtheorem{proposition}[theorem]{Proposition}
\newtheorem{lemma}[theorem]{Lemma}
\newtheorem{corollary}[theorem]{Corollary} \theoremstyle{definition}

\newtheorem{definition}[theorem]{Definition}
\newtheorem{conjecture}[theorem]{Conjecture}
\newtheorem{question}[theorem]{Question}

\newtheorem{example}[theorem]{Example}
\newtheorem*{conventions}{Conventions}

 \theoremstyle{remark}

\DeclareSymbolFont{sfoperators}{OT1}{cmss}{m}{n}
\DeclareSymbolFontAlphabet{\mathsf}{sfoperators} \makeatletter
\def\operator@font{\mathgroup\symsfoperators} \makeatother

\DeclareMathOperator{\Span}{span}

\DeclareMathOperator{\GGG}{G} 
 \DeclareMathOperator{\CCC}{C}
\DeclareMathOperator{\Q}{Q} \DeclareMathOperator{\add}{add}
\DeclareMathOperator{\ann}{ann} \DeclareMathOperator{\Aut}{Aut}
\DeclareMathOperator{\ind}{ind} \DeclareMathOperator{\rad}{rad}
\DeclareMathOperator{\proj}{proj} \DeclareMathOperator{\Fac}{Fac}
\DeclareMathOperator{\Sub}{Sub} 
\DeclareMathOperator{\tors}{tors}
\renewcommand{\Im}{\operatorname{Im}}
\DeclareMathOperator{\torsf}{torf}
\DeclareMathOperator{\ftorsf}{f-torf}
\DeclareMathOperator{\ftors}{f-tors} 
\DeclareMathOperator{\Filt}{Filt} \DeclareMathOperator{\Hom}{Hom} \DeclareMathOperator{\Rad}{Rad}
\DeclareMathOperator{\End}{End} \DeclareMathOperator{\Ext}{Ext}
 
\DeclareMathOperator{\coind}{coind}
\DeclareMathOperator{\Coker}{Coker} 
\DeclareMathOperator{\thick}{thick} \DeclareMathOperator{\silt}{silt}
\DeclareMathOperator{\ctilt}{c-tilt}
\DeclareMathOperator{\brick}{brick}
\DeclareMathOperator{\fbrick}{f-brick}

\DeclareMathOperator{\TTT}{T}
\newcommand{\itrigid}{\tau\operatorname{-rigid}}

\DeclareMathOperator{\presilt}{presilt} 
\newcommand{\op}{\mathrm{op}}
\newcommand{\RR}{\mathbb{R}}

\renewcommand{\max}{\operatorname{max}}
\DeclareMathOperator{\ttilt}{\tau-tilt}
\DeclareMathOperator{\sttilt}{s\tau-tilt}
\newcommand{\twosilt}[1]{2_{#1}\operatorname{-silt}}
\newcommand{\ttwosilt}{2\operatorname{-silt}}
\newcommand{\twopresilt}[1]{2_{#1}\operatorname{-presilt}}
\newcommand{\ttop}{\operatorname{top}\nolimits}
\renewcommand{\mod}{\operatorname{mod}\nolimits}

\newcommand{\KKK}[1]{\operatorname{K}^{\mathrm{b}}(\proj #1)}

 \newcommand{\ie}{\emph{i.e.}~}
\newcommand{\eg}{\emph{e.g.}~}

\newcommand{\ZZ}{\mathbb{Z}}

\newcommand{\cat}[1]{\mathcal{#1}} 
\newcommand{\E}{\cat{E}} \newcommand{\C}{\cat{C}}
\newcommand{\F}{\cat{F}} 
 \renewcommand{\S}{\cat{S}}
\newcommand{\T}{\cat{T}} 

\newcommand{\X}{\cat{X}} \newcommand{\Y}{\cat{Y}}

\DeclareMathOperator{\Tr}{Tr}

\newcommand{\Tors}[1]{\operatorname{T}(#1)}

 \newcommand{\xla}[1]{\xleftarrow{#1}}
\newcommand{\xto}[1]{\xrightarrow{#1}} \newcommand{\set}[1]{\{#1\}}
\newcommand{\setP}[2]{\{#1\mid #2\}} \newcommand{\lperp}{{}^{\perp}}

\newcommand{\gvec}[1]{%
  \begin{pmatrix}%
	#1%
  \end{pmatrix}%
}

\newsavebox\locboxinminipage \newlength\locboxinminipagel
\newcommand{\boxinminipage}[1] {%
  \sbox\locboxinminipage{#1}%
  \settowidth\locboxinminipagel{\usebox{\locboxinminipage}}%
  \begin{minipage}{\locboxinminipagel}\usebox{\locboxinminipage}\end{minipage}%
}

\newenvironment{gnew}{\bgroup\color{blue}}{\egroup}

\newenvironment{gold}{\bgroup\color{red}}{\egroup}

\begin{document}

\title[$\tau$-tilting finite algebras]{$\tau$-tilting finite algebras,
 bricks and $g$-vectors}
 
\author[L. Demonet]{Laurent Demonet}
\email{Laurent.Demonet@normalesup.org}
\urladdr{http://www.math.nagoya-u.ac.jp/~demonet/}

\author[O. Iyama]{Osamu Iyama} \email{iyama@math.nagoya-u.ac.jp}
\urladdr{http://www.math.nagoya-u.ac.jp/~iyama/}

\author[G. Jasso]{Gustavo Jasso} \email{gjasso@math.uni-bonn.de}
\urladdr{http://gustavo.jasso.info}

\address[Demonet and Iyama]{Graduate School of Mathematics, Nagoya University.
  Chikusa-ku, Furo-cho, Nagoya, 464-8602, Japan}

\address[Jasso]{Mathematisches Institut, Universit\"at Bonn, Endenicher
  Allee 60, 53115 Bonn, Germany}

\thanks{The authors would like to thank Alexandra Zvonareva for
  pointing out a crucial typo in the statement of
  \th\ref{char-tau-rigid-finite} in a previous version of this
  article. The first named author was partially supported by JSPS
  Grant-in-Aid for Young Scientists (B) 26800008.  The second named
  author was partially supported by JSPS Grant-in-Aid for Scientific
  Research (B) 24340004, (B) 16H03923, (C) 23540045,
(S) 22224001 and (S) 15H05738. The third
  named author was supported by Consejo Nacional de Ciencia y
  Tecnolog\'ia: Beca para estudios en el extranjero no. 310793.}

\subjclass[2010]{18E40 (primary); 16G20 (secondary).}

\begin{abstract}
  The class of support $\tau$-tilting modules was introduced to
  provide a completion of the class of tilting modules from the point
  of view of mutations. In this article we study $\tau$-tilting finite
  algebras, i.e. finite dimensional algebras $A$ with finitely many
  isomorphism classes of indecomposable $\tau$-rigid modules. We show
  that $A$ is $\tau$-tilting finite if and only if every torsion class
  in $\mod A$ is functorially finite.  We observe that cones generated
  by $g$-vectors of indecomposable direct summands of each support
  $\tau$-tilting module form a simplicial complex $\Delta(A)$.  We
  show that if $A$ is $\tau$-tilting finite, then $\Delta(A)$ is
  homeomorphic to an $(n-1)$-dimensional sphere, and moreover the
  partial order on support $\tau$-tilting modules can be recovered
  from the geometry of $\Delta(A)$.  Finally we give a bijection
  between indecomposable $\tau$-rigid $A$-modules and bricks of $A$
  satisfying a certain finiteness condition, which is automatic for
  $\tau$-tilting finite algebras.
\end{abstract}

\maketitle

\section{Introduction}

Let $A$ be a finite dimensional algebra over an arbitrary field $K$. The class of support
$\tau$-tilting $A$-modules was introduced recently in
\cite{adachi_tau-tilting_2014} so as to complete the class of tilting
modules from the viewpoint of mutations.  This class of modules was
extensively studied in \cite{angeleri-hugel_silting_2014} for
infinitely generated modules.  The aim of this article is to study
algebras which only have finitely many basic support $\tau$-tilting
modules up to isomorphism. We also give general results on simplicial
complexes and $g$-vectors associated with support $\tau$-tilting
modules.

Algebras having finitely many isomorphism classes of basic tilting
modules have been successfully investigated, see for example
Riedtmann--Schofield \cite{riedtmann_simplicial_1991}, Unger
\cite{unger_shellability_1999} and Ingalls--Thomas
\cite{ingalls_noncrossing_2009}. In this case, the class of (support)
tilting modules enjoys particularly nice combinatorial properties.
With this motivation, we introduce the following class of algebras
which are our main concern in this article. Let $A$ be a finite
dimensional algebra. We remind the reader that an $A$-module $M$ is
\emph{$\tau$-rigid} if $\Hom_A(M,\tau M)=0$, where $\tau M$ denotes
the Auslander--Reiten translate of $M$; such a $\tau$-rigid $A$-module
$M$ is called \emph{$\tau$-tilting} if the number $|M|$ of
non-isomorphic indecomposable direct summands of $M$ coincides with
the number of isomorphism classes of simple $A$-modules. A
\emph{support $\tau$-tilting $A$-module} is a $\tau$-tilting
$(A/\langle e\rangle)$-module for some idempotent $e\in A$. As an
immediate consequence of the work of a number of authors
\cite{adachi_tau-tilting_2014,koenig_silting_2014,marks_torsion_2015},
isomorphism classes of support $\tau$-tilting $A$-modules are known to
be in bijective correspondence with several sets of important objects
in representation theory, i.e. certain explicitly described classes
of torsion classes, $t$-structures, silting complexes, co-$t$-structures,
simple minded collections and wide subcategories, respectively.

\begin{definition}
  \th\label{tau-rigid_finite_algebra} Let $A$ be a finite dimensional
  algebra. We say that $A$ is \emph{$\tau$-tilting finite} if there
  are only finitely many isomorphism classes of basic $\tau$-tilting
  $A$-modules
\end{definition}

In \th\ref{char-tau-rigid-finite} we give equivalent conditions for an
algebra to be $\tau$-tilting finite.  In particular, $A$ is
$\tau$-tilting finite if and only if there exist only finitely many
isomorphism classes of indecomposable $\tau$-rigid $A$-modules if and
only if there exist only finitely many isomorphism classes of basic
support $\tau$-tilting $A$-modules.

For example, local algebras and representation finite algebras are
$\tau$-tilting finite. Also, $\tau$-tilting finite algebras whose
radical square vanishes are characterized in
\cite{adachi_characterizing_2014}. In addition, preprojective algebras
of Dynkin quivers are known to be $\tau$-tilting finite
\cite{mizuno_classifying_2014}. Note that all these classes of
algebras contain algebras which are of infinite representation
type. Further information can be found in
\cite{adachi_tilting_2015,eisle_reduction_2016,iyama_classifying_2016,MR3428959,kase_weak_2016,MR3490664,zhang_tau-rigid_2012}.

One of the main results in \cite{adachi_tau-tilting_2014} is that the
map
\[
  M\mapsto \Fac M:=\setP{N\in\mod A}{\exists M^n\twoheadrightarrow N}.
\]
induces a bijection between the set $\sttilt A$ of isomorphism classes
of basic support $\tau$-tilting $A$-modules and the set of
functorially finite torsion classes in $\mod A$, see
\cite[Thm. 2.7]{adachi_tau-tilting_2014}.

After recalling basic definitions and results in Section
\ref{sec:preliminaries}, in Section \ref{sec:tau-rigid finite
  algebras} we investigate further the relationship between support
$\tau$-tilting $A$-modules and torsion classes.  The following result
can be regarded as a $\tau$-tilting-theoretic analogue of
Auslander--Reiten's result stating that a finite dimensional algebra
$A$ is representation-finite if and only if every subcategory of
$\mod A$ is functorially finite
\cite[Prop. 1.2]{auslander_applications_1991}.

\begin{theorem}[see \th\ref{all torsion classes are functorially finite}]
  \th\label{simple-main} Let $A$ be a finite dimensional
  algebra. Then, $A$ is $\tau$-tilting finite if and only if every
  torsion class (equivalently, torsion-free class) in $\mod A$ is
  functorially finite.
\end{theorem}

This result is a consequence of the following generalization of
\cite[Theorem. 2.35]{adachi_tau-tilting_2014} to torsion classes which
are not necessarily functorially finite.

\begin{theorem}[see
  \th\ref{main-torsion-classes}]
  Let $A$ be a finite dimensional algebra and $\T$ a functorially
  finite torsion class. Then, if $\S$ is a torsion class such that
  $\T\supsetneq\S$ (resp. such that $\T\subsetneq\S$) in $\T$, then
  there exists a maximal (resp. minimal) functorially finite torsion
  class $\T'$ with the property $\T\supsetneq\T'\supseteq\S$
  (resp. $\T\subsetneq\T'\subseteq\S$).
\end{theorem}

Recall that $X\in\mod A$ is called a \emph{brick} if $\End_A(X)$ is a
division algebra.  In Section \ref{sec:brick-tau-rigid_correspondence}
we investigate the relationship between $\tau$-rigid $A$-modules and
bricks in $\mod A$. Our main result is the following alternative
characterizations of $\tau$-tilting finite algebras, which complements
\th\ref{simple-main}.

\begin{theorem}[see \th\ref{brick-tau-rigid-wide-finite}]
  \th\label{simple-bricks} Let $A$ be a finite dimensional
  algebra. Then, $A$ is $\tau$-tilting finite if and only if there are
  only finitely many isomorphism classes of bricks in $\mod A$.
\end{theorem}

This follows from the brick--$\tau$-rigid correspondence given below
(note that here $A$ is an arbitrary finite dimensional algebra), which
is of independent interest.

\begin{theorem}[see \th\ref{brick-rigid}]
  Let $A$ be a finite dimensional algebra. Then, there is a bijection
  between isomorphism classes of indecomposable $\tau$-rigid modules
  and isomorphism classes of bricks $X$ in $\mod A$ such that the
  smallest torsion class in $\mod A$ containing $X$ is functorially
  finite.
\end{theorem}

The above theorem is known already in the case when $A$ is the
preprojective algebra of Dynkin type, see \cite{iyama_lattice_2016}.

The rest of this article is devoted to study the following natural
question.

\begin{question}
  \th\label{main_question} Which quivers can be realized as the Hasse
  quiver of support $\tau$-tilting modules over a $\tau$-tilting
  finite algebra?
\end{question}

To study \th\ref{main_question}, we introduce the simplicial complex
$\Delta(A)$ and $g$-vectors as basic tools.  It is shown in
\cite{adachi_tau-tilting_2014} that there is a natural bijection
between $\sttilt A$ and the set $\ttwosilt A$ of two-term silting
complexes in $\KKK{A}$.  In Section \ref{sec:simplicial-complex}, we
associate to $A$ an abstract simplicial complex $\Delta(A)$ whose
maximal faces are in bijection with $\ttwosilt A$. After interpreting
the results from \cite{adachi_tau-tilting_2014} in terms of the
simplicial complex $\Delta(A)$, we prove the following result which is
analogous to the main result of \cite{riedtmann_simplicial_1991}. Our
method relies on the notion of shellability, whose usefulness in
tilting theory was already exploited in
\cite{unger_shellability_1999}.

\begin{theorem}[see \th\ref{shellability-sphere}]
  Let $A$ be a $\tau$-tilting finite algebra with $n+1$ simple
  modules. The following statements hold:
  \begin{enumerate}
  \item If $n\geq1$, then $\Delta(A)$ is shellable.
  \item The geometric realization of $\Delta(A)$ is homeomorphic to an
    $n$-dimensional sphere.
  \item If $n\geq 2$, then $\Delta(A)$ is simply connected.
  \end{enumerate}
\end{theorem}

It is worth mentioning that a finite simplicial complex which is
homeomorphic to a sphere is not necessarily shellable
\cite{rudin_unshellable_1958}.

We consider the groupoid $\Delta^{\max}(A)$ whose objects are the
maximal faces of $\Delta(A)$. It is freely generated by isomorphisms
$\alpha\to\beta$ for each $\alpha$ and $\beta$ sharing a simplex of
codimension 1 in $\Delta(A)$. Thus, morphisms in $\Delta^{\max}(A)$
correspond to finite sequences of mutations of support $\tau$-tilting
$A$-modules. A \emph{cycle} in $\Delta^{\max}(A)$ is a morphism with
the same source and target. The \emph{rank} of a cycle $\mu$ is the
codimension in $\Delta(A)$ of the intersection of all the maximal
faces involved in $\mu$.  The following result is analogous to
\cite[Thm. 9.17]{fomin_cluster_2008}.

\begin{theorem}[see \th\ref{property-R2}]
  Let $A$ be a finite dimensional algebra. If $\Delta(A)$ is
  shellable, then every cycle in $\Delta^{\max}(A)$ is generated by
  cycles of rank 2.
\end{theorem}

In Section \ref{sec:g-vectors} we study the $g$-vectors associated to
two-term presilting objects in a $K$-linear $\Hom$-finite Krull-Schmidt
triangulated category $\T$. 
We observe some basic properties of $g$-vectors including sign-coherence
(\th\ref{sign-coherent}), transitivity of $G$-matrices
(\th\ref{G-identities}), and the fact that $g$-vectors determine
two-term presilting objects
(\th\ref{order-implies-epi}\eqref{it:g-vectors_determine_presilting}). 
Our results do not require any restriction on the base field $K$ thanks
to general observations on field extensions (\th\ref{field-extension}).
Since the $g$-vectors of indecomposable direct summands of a fixed basic
two-term silting object $M$ gives a basis of the
Grothendieck group $K_0(\T)$ of $\T$ (see
\th\ref{silting-object-gives-basis-of-K0}\eqref{it:S-K0-basis}), we
can naturally associate to $M$ a cone $C(M)$ in $K_0(\T)$. Then we
prove the following result.

\begin{theorem}[see \th\ref{cones}]
  Let $M$ and $N$ be two-term silting objects in $\T$. Then, the
  following statements hold:
  \begin{enumerate}
  \item Let $X\in\T$ be an object such that
    $\add X = \add M\cap\add N$. Then, we have $C(M)\cap C(N)=C(X)$.
  \item If $M\not\cong N$, then $C(M)$ and $C(N)$ intersect only at
    their boundaries.
  \end{enumerate}
\end{theorem}

Consequently, the $g$-vectors of presilting objects give a natural
geometric realization of the simplicial complex $\Delta(A)$. This is
analogous to known results for tilting modules
\cite{hille_volume_2006} and cluster-tilting objects
\cite{plamondon_generic_2013}. We also show a strong connection
between the partial order on two-term silting objects in $\T$ and
their $g$-vectors by proving the following theorem.

\begin{theorem}[see \th\ref{cones_determine_partial_order}]
  Let $A$ be a $\tau$-tilting finite algebra. Then the partial order
  of support $\tau$-tilting $A$-modules is entirely determined by the
  cones of $g$-vectors.
\end{theorem}

The following general question is interesting from a combinatorial
viewpoint.

\begin{question}
  Which sets of integer vectors can be realized as the set of
  $g$-vectors of indecomposable $\tau$-rigid modules over a
  $\tau$-tilting finite algebra?
\end{question}

At the end of Section \ref{sec:g-vectors} we provide some examples
and we explain a connection to cluster-tilting theory.

\begin{conventions}
  Let $K$ be an arbitrary field.  Throughout this article $A$ always denotes a
  finite dimensional $K$-algebra.  We denote the category of finite
  dimensional right $A$-modules by $\mod A$, and the Auslander--Reiten
  translation by $\tau$.  For $M\in \mod A$, we denote the number of
  pairwise non-isomorphic indecomposable direct summands of $M$ by
  $|M|$. For example, $|A|$ equals the number of simple $A$-modules.
  We denote the full subcategory of $\mod A$ consisting of all modules
  which are direct sums of direct summands of $M$ by $\add M$.  We
  denote by $\Fac M$ the full subcategory of $\mod A$ consisting of
  all $A$-modules which are generated by a finite direct sum of copies
  of $M$.  Dually, we denote by $\Sub M$ the full subcategory of
  $\mod A$ consisting of all $A$-modules which are cogenerated by a
  finite direct sum of copies of $M$. If $(P,\leqslant)$ is a
  partially ordered set and $x,y\in P$, then the (closed) interval
  between $x$ and $y$ in $P$ is the subset of $P$ given by
  \[
    [x,y] :=\setP{z\in P}{x\leqslant z\leqslant y}.
  \]
  When we write an object $M$ in a Krull-Schmidt category as
  $M=M_1\oplus\cdots\oplus M_n$ we always mean that $M$ is basic and
  that each one of the $M_i$ is indecomposable.
\end{conventions}

\section{Preliminaries}
\label{sec:preliminaries}

Let $\C$ be a small additive category and $\X$ be a subcategory of
$\C$.  We say that $\X$ is \emph{contravariantly finite in $\C$} if
for every object $M$ of $\C$ there exist an object $X$ of $\X$ and a
morphism $f:X\to M$ such that the sequence of functors
\[
  \begin{tikzcd}
    \C(-,X)\rar{f\circ?}&\C(-,M)\rar&0
  \end{tikzcd}
\]
is exact. Dually, we say that $\X$ is \emph{covariantly finite in
  $\C$} if for every object $M$ of $\C$ there exist an object $X$ of
$\X$ and a morphism $g\colon M\to X$ such that the sequence of
functors
\[
  \begin{tikzcd}
    \C(X,-)\rar{?\circ g}&\C(M,-)\rar&0
  \end{tikzcd}
\]
is exact.  We say that $\X$ is \emph{functorially finite in $\C$} if
$\X$ is both contravariantly and covariantly finite in $\C$.

\subsection{Torsion pairs and $\tau$-tilting theory}

Let $\T,\F$ be full subcategories of $\mod A$.  We say that $\T$ is a
\emph{torsion class} if it is closed under factor modules and
extensions in $\mod A$.  We say that $\F$ is a \emph{torsion-free
  class} if it is closed under submodules and extensions in $\mod A$.
We say that $(\T,\F)$ is a \emph{torsion pair} if $\T$ is a torsion
class, $\F$ is a torsion-free class, and $\F=\T^\perp$ (or,
equivalently, $\T={}^\perp\F$). We denote the set of all torsion
(resp. torsion-free) classes in $\mod A$ by $\tors A$ (resp.
$\torsf A$).  Accordingly, we denote the set of functorially finite
torsion (resp. torsion-free) classes in $\mod A$ by $\ftors A$
(resp. $\ftorsf A$). We say that an $A$-module $M\in\T$ is
\emph{$\Ext$-projective in $\T$} if for all $N\in\T$ we have
$\Ext_A^1(M,N)=0$.

\begin{proposition}
  \th\label{characterization-of-ftors}
  \cite{auslander_almost_1981,hoshino_tilting_1982,smalo_torsion_1984}
  Let $A$ be a finite dimensional algebra and $(\T,\F)$ a torsion pair
  in $\mod A$. The following statements are equivalent:
  \begin{enumerate}
  \item\label{it:T-is-ff} The torsion class $\T$ is functorially
    finite.
  \item The torsion-free class $\F$ is functorially finite.
  \item\label{it:Ext-projective} There exist a basic $A$-module
    $P(\T)\in \T$ such that $\Fac P(\T) = \T$ and $\add P(\T)$
    coincides with the class of $\Ext$-projective $A$-modules in $\T$.
  \end{enumerate}
  If any of the above equivalent conditions hold, then the $A$-module
  $P(\T)$ is a tilting $(A/\ann \T)$-module.
\end{proposition}

We remind the reader of the definition of support $\tau$-tilting pairs
and support $\tau$-tilting modules.

\begin{definition}
  \th\label{def-sttilt} \cite[Def. 0.1 and
  0.3]{adachi_tau-tilting_2014} Let $A$ be a finite dimensional
  algebra.
  \begin{enumerate}
  \item A pair $(M,P)\in(\mod A)\times(\proj A)$ is
    \emph{$\tau$-rigid} if $\Hom_A(M,\tau M)=0$ and
    $\Hom_A(P,M)=0$. In this case we say that $M$ is a
    \emph{$\tau$-rigid $A$-module}.
  \item A $\tau$-rigid pair $(M,P)$ is \emph{support $\tau$-tilting}
    if $|M|+|P|=|A|$.  In this case we say that $M$ is a \emph{support
      $\tau$-tilting $A$-module}. If $P=0$ we say that $M$ is a
    \emph{$\tau$-tilting $A$-module}.
  \end{enumerate}
  For convenience, we denote the set of isomorphism classes of basic
  $\tau$-tilting (resp. support $\tau$-tilting) $A$-modules by
  $\ttilt A$ (resp. $\sttilt A$).
\end{definition}

The following result collects the basic properties of support
$\tau$-tilting modules.

\begin{theorem}
  \th\label{tau-tilting-theory} \cite[Thms. 2.7, 2.10, 2.12 and
  2.18]{adachi_tau-tilting_2014} Let $A$ be a finite dimensional
  algebra. The following statements hold:
  \begin{enumerate}
  \item\label{it:sttiltA-torsA} The map $M\mapsto \Fac M$ induces a
    bijection
    \[
      \begin{tikzcd}
        \sttilt A \rar[leftrightarrow]&\ftors A
      \end{tikzcd}
    \]
    whose inverse is given by $\T\mapsto P(\T)$.
  \item\label{it:Bongartz-completion} Let $U$ be a $\tau$-rigid
    $A$-module.  Then, $\lperp(\tau U)$ is a functorially finite
    torsion class in $\mod A$ and $U\in\add P(\lperp{(\tau U)})$.
    Moreover, $P(\lperp{(\tau U)})$ is a $\tau$-tilting $A$-module
    which we call the \emph{Bongartz completion of $U$}.
  \item\label{it:maximal-tau-rigid} Let $M$ be a $\tau$-rigid
    $A$-module.  Then, $M$ is a $\tau$-tilting $A$-module if and only
    if for every $A$-module $N$ such that $M\oplus N$ is $\tau$-rigid
    we have $N \in\add M$.
  \item\label{it:mutation-in-sttiltA} Let $(M,P)$ be a basic
    $\tau$-rigid pair of $A$-modules (\ie $M$ and $P$ are basic
    $A$-modules) such that $|M|+|P|=|A|-1$. Then, there exist exactly
    two basic support $\tau$-tilting pairs $(M_i,P_i)$ ($i=1,2$) of
    $A$-modules such that $M$ and $P$ are direct summands of $M_i$ and
    $P_i$ respectively. In this case, we say that $(M_1,P_1)$ and
    $(M_2,P_2)$ are \emph{mutation of each other}.
  \end{enumerate}
\end{theorem}

The following definitions are suggested by
\th\ref{tau-tilting-theory}.

\begin{definition}
  \th\label{QsttiltA} Let $A$ be a finite dimensional algebra. The set
  $\sttilt A$ has a natural partial order defined as follows: For
  $M,N\in\sttilt A$, we define $M\geqslant N$ if
  $\Fac M\supseteq \Fac N$ or, equivalently, there exist an
  epimorphism $M^k\to N$ for some $k>0$.  We denote the Hasse quiver
  of $\sttilt A$ by $Q(\sttilt A)$.
\end{definition}

\begin{proposition}
  \th\label{arrows-mutations}
  \cite[Cor. 2.31]{adachi_tau-tilting_2014} Arrows in the Hasse quiver
  $Q(\sttilt A)$ correspond to mutations.
\end{proposition}

As a consequence of \th\ref{arrows-mutations} and
\th\ref{tau-tilting-theory}\eqref{it:mutation-in-sttiltA}, the
underlying graph of $Q(\sttilt A)$ is $|A|$-regular.

\subsection{Silting objects and support $\tau$-tilting modules}

We remind the reader of the definition of a silting object in a
triangulated category. Let $\T$ be a $K$-linear, $\Hom$-finite,
Krull-Schmidt, triangulated category with suspension functor $\Sigma$.

\begin{definition}
  Let $S\in\T$.
  \begin{enumerate}
  \item We say that $S$ is \emph{presilting} if for all $k>0$ we have
    $\T(S,\Sigma^k S)=0$.
  \item We say that $S$ is \emph{silting} if $S$ is presilting and
    $\T=\thick S$.
  \end{enumerate}
  We denote the set of isomorphism classes of basic silting
  (resp. presilting) objects in $\T$ by $\silt\T$
  (resp. $\presilt\T$).
\end{definition}

Following \cite[Def. 2.10, Thm. 2.11]{aihara_silting_2012}, we define
a partial order on $\silt\T$ by declaring $M\geq N$ if and only if for
all $k>0$ we have $\T(M,\Sigma^k N)=0$.

If $\X,\Y$ are full subcategories of $\T$ which are closed under
direct sums and direct summands, then we denote by $\X*\Y$ the full
subcategory of $\T$ given by all objects $Z\in\T$ such that there
exist $X\in\X$, $Y\in\Y$ and a triangle $X\to Z\to Y\to \Sigma X$.
Note that if $\T(\X,\Y)=0$, then $\X\ast\Y$ is closed under direct
summands, see \cite[Prop. 2.1]{iyama_mutation_2008}.

\begin{theorem}
  \th\label{silting-object-gives-basis-of-K0} Let $\T$ be a
  $K$-linear, $\Hom$-finite, Krull-Schmidt, triangulated category and
  $S=S_1\oplus S_2\oplus\cdots\oplus S_n$ a silting object in
  $\T$. Then, the following statements hold:
  \begin{enumerate}
  \item\label{it:S-K0-basis} \cite[Thm. 2.27]{aihara_silting_2012} The
    Grothendieck group of $\T$ has a basis $\set{[S_1],\dots,[S_n]}$.
  \item\label{it:bijection-sttilt-twosilt}
    \cite[Thm. 0.2]{iyama_intermediate_2014} Let
    $A:=\End_\T(S)$. Then, the map $M\mapsto \T(S,M)$ induces an
    order-preserving bijection
    \[
      \begin{tikzcd}
        \twosilt{S}\T:=\setP{M\in\silt \T}{M\in (\add S)*(\add\Sigma
          S)}\rar[leftrightarrow]&\sttilt A.
      \end{tikzcd}
    \]
  \end{enumerate}
\end{theorem}

For later use, we define
$\twopresilt{S}{\T}:=(\presilt\T)\cap((\add S)*(\add\Sigma S))$. For a
finite dimensional algebra $A$, we define
$\ttwosilt{A}:=\twosilt{A}{(\KKK{A})}$. Thus, $\ttwosilt{A}$ consists
of the silting complexes in $\KKK{A}$ which are concentrated in
cohomological degrees $-1$ and $0$. Therefore we call the objects of
$\twosilt{S}{\T}$ (resp. $\twopresilt{S}{\T}$) \emph{two-term silting
  objects with respect to $S$} (resp. \emph{two-term presilting
  objects with respect to $S$}). The following is an immediate
consequence of
\th\ref{silting-object-gives-basis-of-K0}\eqref{it:bijection-sttilt-twosilt}.

\begin{corollary}
  Let $\T$ be a $K$-linear, $\Hom$-finite, Krull-Schmidt, triangulated
  category and $S$ a silting object in $\T$. Set
  $A:=\End_\T(S)$. Then, there are order preserving bijections
  \[
    \begin{tikzcd}
      \twosilt{S}{\T}\rar[leftrightarrow]&\sttilt{A}\rar[leftrightarrow]&\ttwosilt{A}.
    \end{tikzcd}
  \]
\end{corollary}

\begin{corollary}
  \th\label{char-tau-rigid-finite} Let $A$ be a finite dimensional
  algebra. Then, each of the following conditions is equivalent to $A$
  being $\tau$-tilting finite:
  \begin{enumerate}
  \item There are only finitely many isomorphism classes of
    indecomposable $\tau$-rigid $A$-modules.
  \item There are only finitely many isomorphism classes of
    indecomposable two-term presilting complexes in $\KKK{A}$.
  \item The set $\ttwosilt{A}$ is finite.
  \item The set $\ftors A$ is finite.
  \item The set $\ftorsf A$ is finite.
  \end{enumerate}
\end{corollary}

\section{Main results on support $\tau$-tilting modules and torsion pairs}

\label{sec:tau-rigid finite algebras}

\subsection{A general result on mutation of torsion classes}

The aim of this subsection is to prove the following general result,
which is an extension of \cite[Thm. 2.35]{adachi_tau-tilting_2014} to
torsion classes which are not necessarily functorially finite, and
therefore we need a completely different argument to prove it. It is
our main tool for proving the characterization of $\tau$-tilting
finite algebra given in \th\ref{all torsion classes are functorially
  finite}.

\begin{theorem}
  \th\label{main-torsion-classes} Let $A$ be a finite dimensional
  algebra and $M$ a support $\tau$-tilting $A$-module. Then, the
  following statements hold:
  \begin{enumerate}
  \item\label{it:main-torsion-classes-a} Let $\T$ be a torsion class
    in $\mod A$ such that $\Fac M\supsetneq\T$.  Then, there exists
    $N\in\sttilt A$ satisfying the following conditions:
    \begin{itemize}
    \item The support $\tau$-tilting $A$-modules $M$ and $N$ are
      mutation of each other.
    \item We have $\Fac M\supsetneq\Fac N\supset\T$.
    \end{itemize}

  \item\label{it:main-torsion-classes-b} Let $\T$ be a torsion class
    in $\mod A$ such that $\Fac M\subsetneq\T$.  Then, there exists
    $L\in\sttilt A$ satisfying the following conditions:
    \begin{itemize}
    \item The support $\tau$-tilting $A$-modules $M$ and $L$ are
      mutation of each other.
    \item We have $\Fac M\subsetneq \Fac L\subset\T$.
    \end{itemize}
  \end{enumerate}
\end{theorem}

Let us begin with a technical result. Let $\E$ and $\F$ be exact
categories. An \emph{equivalence of exact categories} is an
equivalence $F\colon\E\to\F$ of additive categories such that a
complex $0\to X\to Y\to Z\to0$ in $\E$ is an admissible exact sequence
if and only if $0\to FX\to FY\to FZ\to0$ is an admissible exact
sequence in $\F$.

\begin{proposition}
  \th\label{preparation1} Let $A$ be a finite dimensional algebra, $M$
  a support $\tau$-tilting $A$-module and $B:=\End_A(M)$.  Let
  \[
    F:=\Hom_A(M,-):\mod A\to\mod B.
  \]
  Then, the following statements hold:
  \begin{enumerate}
  \item\label{it:prep1-c} If $f:X\to Y$ is a morphism in $\Fac M$ such
    that $Ff$ is surjective, then $f$ is surjective.
  \item\label{it:prep1-a} The functor $F$ induces an equivalence of
    exact categories $F|_{\Fac M}\colon\Fac M\xto{\sim}\Sub DM$.
  \item\label{it:prep1-b} If $\T$ is a torsion class in $\mod A$
    contained in $\Fac M$, then $F\T$ is closed under extensions in
    $\mod B$.
  \end{enumerate}
\end{proposition}
\begin{proof}
  \eqref{it:prep1-c} Let $X\xto{f}Y\xto{g}Z\to0$ be an exact sequence
  such that $Ff$ is surjective.  Since $\Im f\in\Fac M$ and $\Hom_A(M,\tau M)=0$,
  %and $M$ is $\Ext$-projective in $\Fac M$ 
  %by \th\ref{tau-tilting-theory}\eqref{it:sttiltA-torsA}
  we have $\Ext^1_A(M,\Im f)=D\overline{\Hom}_A(\Im f,\tau M)=0$.
  By applying $F$ we obtain a complex
  \[
    \begin{tikzcd}
      FX\rar{Ff}&FY\rar{Fg}&FZ\rar&0
    \end{tikzcd}
  \]
  which is exact at $FZ$.  Since $Ff$ is surjective by assumption, we
  have $FZ=0$.  Finally, since $Z\in\Fac M$ we have that $Z=0$.

  \eqref{it:prep1-a} By the last part of
  \th\ref{characterization-of-ftors} we have that $M$ is a tilting
  $(A/\ann M)$-module and clearly we have
  $\Fac M\subset\mod (A/\ann M)$. Since $F=\Hom_{A/\ann M}(M,-)$ on
  $\Fac M$, we have that $F\colon\Fac M\to\Sub DM$ is an equivalence 
  by Brenner--Butler's tilting theorem
  \cite[Thm. VI.3.8]{assem_elements_2006}, with quasi-inverse
  $$G:=-\otimes_BM\colon\Sub DM\to\Fac M.$$
  Moreover, since $\Ext^1_A(M,\Fac M)=0$
  and ${\rm Tor}^B_1(\Sub DM,M)=0$, $F$ and $G$ induce bijections between
  %$F$ send
  short exact sequences in $\Fac M$ and those in $\Sub DM$. Thus the assertion follows.

  %Now we show that $F|_{\Fac M}\colon\Fac M\to\Sub DM$ is an
  %equivalence of exact categories. 
  %Since $F(\Fac M)=\Sub DM$ is closed
  %under extensions in $\mod B$, we only have to show the following:
  %\begin{quote}
  %  If $0\to FX\xto{Ff}FY\xto{Fg}FZ\to0$ is an exact sequence in
  %  $\mod B$ such that $X,Y,Z\in\Fac M$, then the sequence
  %  $0\to X\to Y\to Z\to0$ is exact in $\mod A$.
  %\end{quote}
  %Since $F(fg)=0$ and $X\in\Fac M$, we have $fg=0$.  Thus the sequence
  %\begin{equation}\label{eq:sequence in mod A}
  %  0\to X\xto{f}Y\xto{g}Z\to0
  %\end{equation}
  %is a complex in $\mod A$.  It follows from part \eqref{it:prep1-c}
  %that $g$ is surjective.  Let $W:=\Ker g$ and $f':X\to W$ be the
  %morphism induced by $f$. Then we have an exact sequence
  %\[
  %  \begin{tikzcd}
  %    FY\rar{Fg}&FZ\rar&\Ext_A^1(M,W)\rar&\Ext_A^1(M,Y)=0
  %  \end{tikzcd}
  %\]
  %where $Fg$ is an epimorphism, hence $\Ext_A^1(M,W)=0$. Since both
  %$FX$ and $FW$ are kernels of $Fg$, it follows that $Ff':FX\to FW$ is
  %an isomorphism. Hence $f'$ is an isomorphism, so the sequence
  %$0\to X\to Y\to Z\to 0$ is exact in $\mod A$.

  Part \eqref{it:prep1-b} now follows immediately from
  \eqref{it:prep1-a}.
\end{proof}

Let $\C$ be a full subcategory of $\mod A$. We define $\Tors{\C}$ to
be the smallest torsion class containing $\C$. Thus,
\[
  \Tors{\C} = \bigcap_{\C\subseteq\T\in\tors A} \T.
\]
Also, we define $\Filt\C$ to be the full subcategory of $\mod A$ whose
objects are the $A$-modules $M$ having a finite filtration
\[
  0= M_0\subset M_1\subset \cdots\subset M_t=M
\]
such that for all $i\in\set{0,1,\dots,t}$ we have
$M_{i+1}/M_i\in\C$.
The following observation follows easily from the definitions.

\begin{proposition}
  \th\label{describe tors} Let $A$ be a finite dimensional
  algebra. For every subcategory $\C$ of $\mod A$ we have
  $\Tors{\C}=\Filt(\Fac\C)$.
\end{proposition}

Note that $\Fac(\Filt\C)$ is not necessarily a torsion class, since it
need not be closed under extensions.

\begin{theorem}\th\label{recover}
  Let $A$ be a finite dimensional algebra, $M$ a support
  $\tau$-tilting $A$-module, $B:=\End_A(M)$, and
  $F\colon \T_0:=\Fac M\xto{\sim}\Y_0:=\Sub DM$ the equivalence given
  in \th\ref{preparation1}\eqref{it:prep1-a}. Then, the following
  statements hold:
  \begin{enumerate}
  \item\label{it:recover-a} The map
    \[
      \setP{\T\in\tors A}{\T\subset\T_0}\to\tors B
    \]
    given by $\T\mapsto\Tors{F\T}$ is injective.
  \item\label{it:recover-b} For each $\T\in\tors A$ with
    $\T\subset\T_0$, we have $F\T=\Y_0\cap\Tors{F\T}$.
  \end{enumerate}
\end{theorem}
\begin{proof}
  Since $F$ is an equivalence, we only need to prove part
  \eqref{it:recover-b}; for which is sufficient to show that
  $\Y_0\cap\Tors{F\T}\subset F\T$.  By
  \th\ref{preparation1}\eqref{it:prep1-c}, we have
  \begin{equation}
    \label{eq:first step}
    \Y_0\cap\Fac(F\T)=F\T.
  \end{equation}
  Thus we only have to show that $\Y_0\cap\Tors{F\T}\subset\Fac(F\T)$.

  (i) Let $0\to X\to Y\to Z\to0$ be an exact sequence with
  $X,Z\in\Fac(F\T)$ and $Y\in\Y_0$. We will show that $Y\in\Fac(F\T)$.

  Since $Y\in\Y_0=\Sub DM$ which is obviously closed under submodules,
  we have $X\in\Y_0$.  Thus $X\in\Y_0\cap\Fac(F\T)=F\T$ by
  \eqref{eq:first step}.  Take a surjection $f:Z'\to Z$ with
  $Z'\in F\T$ and consider a pull-back diagram
  \begin{center}
    \begin{tikzcd}
      0 \rar & X \rar\dar[equals] & Y' \rar \dar{g} & Z' \rar\dar{f} & 0\\
      0 \rar & X \rar & Y \rar & Z \rar & 0
    \end{tikzcd}
  \end{center}
  Since $F\T$ is extension-closed by
  \th\ref{preparation1}\eqref{it:prep1-b}, we have $Y'\in F\T$.  By
  the Five lemma we have that $g$ is surjective, hence $Y\in\Fac(F\T)$
  as required.

  (ii) Now we are ready to show that
  $\Y_0\cap\Tors{F\T}\subset\Fac(F\T)$.  Let
  $X\in \Y_0\cap\Tors{F\T}$. We will show that $X\in\Fac(F\T)$. It
  follows from \th\ref{describe tors} that $X$ has a finite filtration
  $0=X_0\subset X_1\subset\cdots\subset X_t=X$ such that
  $X_{i+1}/X_i\in\Fac(F\T)$ for all $i$.  Clearly we can assume
  $t\geq2$.  Since $X\in\Y_0=\Sub DM$, we have $X_i\in\Y_0$ for all
  $i$.  Thus the exact sequence $0\to X_1\to X_2\to X_2/X_1\to0$
  satisfies all the assumptions in (i) and we have $X_2\in\Fac(F\T)$.
  By induction on the length of filtrations, we conclude.
\end{proof}

As an application of \th\ref{recover}, we prove the following
observations which will not be used in this article. Note that they
are known for functorially finite torsion classes, see
\cite[Thm. 2.33]{adachi_tau-tilting_2014}.

\begin{example}
  Let $A$ be a finite dimensional algebra. The following statements
  hold: \th\label{nothing between}
  \begin{enumerate}
  \item\label{it:nothing between-maximal torsion class} Let $e\in A$
    be a primitive idempotent.  There are no torsion classes in
    $\mod A$ between $\mod A$ and $\Fac((1-e)A)$.
  \item\label{it:nothing between-b} Let $M$ and $N$ be support
    $\tau$-tilting $A$-modules.  If $M$ and $N$ are mutation of each
    other, then there are no torsion classes in $\mod A$ between
    $\Fac M$ and $\Fac N$.
  \end{enumerate}
\end{example}
\begin{proof}
  \eqref{it:nothing between-maximal torsion class} Let $S:=\ttop(eA)$.
  Then $\Fac((1-e)A)$ consists of all $A$-modules $X$ such that
  $S\notin\add(\ttop X)$.  If $\T$ is a torsion class in $\mod A$
  satisfying $\mod A\supseteq\T\supsetneq\Fac((1-e)A)$, then there
  exists $X\in\T$ such that $S\in\add(\ttop X)$.  Since $\T$ is a
  closed under factor modules, we have $S\in\T$.  Since any $A$-module
  $Y$ has a submodule $Z$ such that $Z\in\Fac((1-e)A)$ and
  $Y/Z\in\add S$, we have $Y\in\T$. Thus $\mod A\subseteq \T$, which
  is what we needed to show.
  
  \eqref{it:nothing between-b} Let $\T$ be a torsion class such that
  $\Fac M\supseteq\T\supsetneq\Fac N$.  By
  \th\ref{recover}\eqref{it:recover-a}, we have
  \[
    \mod B=\Tors{F(\Fac M)}\supseteq\Tors{F\T}\supsetneq\Tors{F(\Fac
      N)}.
  \]
  Since $(1-e_k)B\in\add FN$ where $k$ is given by $N=\mu_k(M)$, we
  have $\Tors{F(\Fac N)}\supseteq\Fac((1-e_k)B)$.  Thus
  $\mod B\supseteq\Tors{F\T}\supsetneq\Fac((1-e_k)B)$. By
  \eqref{it:nothing between-maximal torsion class} applied to the
  finite dimensional algebra $B$ we have $\mod B=\Tors{F\T}$ and,
  again by \th\ref{recover}\eqref{it:recover-a}, we have that
  $\Fac M=\T$ as required.
\end{proof}

Now we have a criterion to decide when a mutation of support
$\tau$-tilting modules becomes larger in terms of the partial order of
$\sttilt A$.

\begin{proposition}
  \th\label{bigger or smaller} Let $A$ be a finite dimensional
  algebra, $M$ a support $\tau$-tilting $A$-module, and let
  $B:=\End_A(M)$, and $F\colon \T_0:=\Fac M\xto{\sim}\Y_0:=\Sub DM$
  the equivalence given in \th\ref{preparation1}\eqref{it:prep1-a}.
  Let $M_k$ an indecomposable summand of $M$.
  \begin{enumerate}
  \item\label{it:muk-a-37} We have $M>\mu_k(M)$ if and only if
    $S_k\in \Y_0$, where $S_k$ is the simple $B$-module corresponding
    to the summand $M_k$ of $M$.
  \item\label{it:muk-b-37} Let $\T\in\tors A$ with $\T\subset\T_0$ and
    assume $S_k\in \Y_0\setminus F\T$. Then
    $\T_0\supsetneq \Fac\mu_k(M)\supset\T$ holds.
  \end{enumerate}
\end{proposition}

\begin{proof}
  \eqref{it:muk-a-37} Let
  \begin{equation}
    \label{eq:right approximation}
    \begin{tikzcd}
      M'\rar{f}&M_k\rar{g}&C_k\rar&0
    \end{tikzcd}
  \end{equation}
  be an exact sequence with a right $(\add(M/M_k))$-approximation $f$
  of $M_k$.  Then, $M>\mu_k(M)$ if and only if $M_k\notin\Fac(M/M_k)$
  if and only if $C_k\neq0$ (see
  \cite[Thm. 2.30]{adachi_tau-tilting_2014}).

  (i) First we show that $C_k\neq0$ implies $S_k\in \Y_0$.  By
  applying $F$ to \eqref{eq:right approximation}, we obtain a complex
  \[
    \begin{tikzcd}
      FM'\rar{Ff}&FM_k\rar{Fg}&FC_k\rar&0
    \end{tikzcd}
  \]
  where $Fg$ is surjective since we have $\Ext^1_A(M,\Im f)=0$ by
  \th\ref{tau-tilting-theory}\eqref{it:sttiltA-torsA}.  In particular,
  the $B$-module $FC_k$ is a factor module of $\Coker Ff$. Since $Ff$
  is a right $(\add(B/FM_k))$-approximation of an indecomposable
  projective $B$-module $FM_k$, every composition factor of
  $\Coker Ff$ is isomorphic to $S_k$.  Thus every composition factor
  of $FC_k$ is isomorphic to $S_k$.  Since $FC_k\neq0$ and
  $\Y_0=\Sub DM$ is a torsion-free class in $\mod B$, we have
  $S_k\in \Y_0$.

  (ii) Now we show that $S_k\in \Y_0$ implies $C_k\neq0$.  Let
  $S_k=FX$ with $X\in\T_0$.  Then the natural surjection
  $FM_k=e_k B\to S_k=FX$ is of the form $Fh$ for some non-zero
  $h\in\Hom_A(M_k,X)$.  Since the composition $(Ff)(Fh):FM'\to FX=S_k$
  vanishes, we have $fh=0$.  Since $h$ is non-zero, $f$ is not
  surjective.  Thus $C_k\neq0$. This finishes the proof.

  \eqref{it:muk-b-37} Since $S_k\notin F\T$, we have
  $F\T\subset\Fac((1-e_k)B)$ by
  \th\ref{preparation1}\eqref{it:prep1-c}.  Thus we have
  \[
    F\T\subset\Fac((1-e_k)B)\subset\Tors{F(\Fac\mu_k(M))}.
  \]
  Finally, by \th\ref{recover}\eqref{it:recover-b}, we have
  \[
    F\T\subset \Y_0\cap \Tors{F(\Fac\mu_k(M))=F(\Fac\mu_k(M))}
  \]
  Thus $\T\subset\Fac\mu_k(M)$.
\end{proof}

We need the following observation.

\begin{lemma}
  \th\label{general-observation} Let $A$ be a finite dimensional
  algebra. Let $M=M_1\oplus\cdots\oplus M_n$ be a basic $A$-module
  with indecomposable direct summands $M_i$, and
  $M'_k\xto{f_k}M_k\to C_k\to0$ be an exact sequence with a right
  $(\add(M/M_k))$-approximation $f_k$ of $M_k$.  Then
  $M\in\Tors{C_1\oplus\cdots\oplus C_n}$.
\end{lemma}
\begin{proof}
  We define $A$-modules $N_i\in\add M$ and $E_i$ inductively by
  $N_0:=M$ and an exact sequence
  \[
    N_{i+1}\xto{g_i}N_i\to E_i\to0
  \]
  where $N_i:=\bigoplus_{k=1}^n M_k^{\oplus a_{ik}}$ and
  $g_i=\bigoplus_{k=1}^nf_k^{\oplus a_{ik}}$ for some $a_{ik}>0$.  It
  readily follows that $E_i\in\add(C_1\oplus\cdots\oplus C_n)$.

  Take a sufficiently large $\ell$ such that
  $(\rad \End_A(M))^\ell=0$.  Then $g_1\cdots g_\ell=0$ holds, and
  hence $M$ is filtered by modules in
  $\Fac(C_1\oplus\cdots\oplus C_n)$.  Thus the assertion follows.
\end{proof}

We are ready to give the proof of \th\ref{main-torsion-classes}.

\begin{proof}[Proof \th\ref{main-torsion-classes}]
  We only prove part \eqref{it:main-torsion-classes-a}, since part
  \eqref{it:main-torsion-classes-b} is analogous.

  {\it Case 1:} Assume that there exists a simple $B$-module $S_k$ in
  $F(\Fac M)\setminus F\T$. Then
  $\Fac M\supsetneq \Fac\mu_k(M)\supset\T$ holds by \th\ref{bigger or
    smaller}\eqref{it:muk-b-37}, and the assertion follows.

  {\it Case 2:} Assume that all simple $B$-modules in $F(\Fac M)$ are
  contained in $F\T$.  For each indecomposable summand $M_k$ of $M$ we
  take an exact sequence
  \[
    M'_k\xto{f_k}M_k\to C_k\to0
  \]
  with a right $(\add(M/M_k))$-approximation $f_k$ of $M_k$.

  We show that $FC_k\in\Tors{F\T}$ holds for all $k$.  We can assume
  $C_k\neq0$. Any composition factor of the $B$-module $FC_k$ is
  $S_k$ as we discussed in the proof of \th\ref{bigger or smaller}.
  Since $F(\Fac M)=\Sub DM$ is a torsion-free class, we have
  $S_k\in F(\Fac M)$.  By our assumption $S_k\in F\T$, which implies
  $FC_k\in\Tors{F\T}$.

  Since $FC_k\in F(\Fac M)\cap\Tors{F\T}=F\T$ by
  \th\ref{recover}\eqref{it:recover-b}, we have $C_k\in\T$.  By
  \th\ref{general-observation}, we have $M\in\T$, a contradiction to
  $\Fac M\supsetneq\T$.
\end{proof}

\subsection{$\tau$-tilting finite algebras}

Let $A$ be a finite dimensional
algebra. \th\ref{tau-tilting-theory}\eqref{it:sttiltA-torsA} shows
that support $\tau$-tilting $A$-modules parametrize the torsion
classes in $\mod A$ which are functorially finite. It is then natural
to characterize those algebras for which all torsion classes are
functorially finite. The following is the main result of this
subsection; it shows that these algebras are precisely the
$\tau$-tilting finite algebras, see \th\ref{tau-rigid_finite_algebra}.

\begin{theorem}
  \th\label{all torsion classes are functorially finite} Let $A$ be a
  finite dimensional algebra. Then, the following conditions are
  equivalent:
  \begin{enumerate}
  \item\label{it:all tc are ff-a} The algebra $A$ is $\tau$-tilting
    finite.
  \item\label{it:all tc are ff-b} Every torsion class in $\mod A$ is
    functorially finite.
  \item\label{it:all tc are ff-c} Every torsion-free class in $\mod A$
    is functorially finite.
  \end{enumerate}
  In this case, the sets $\tors A$ and $\torsf A$ are finite.
\end{theorem}

Before giving the proof of \th\ref{all torsion classes are
  functorially finite} we prove the following natural characterization
of $\tau$-tilting finite algebras.

\begin{proposition}
  \th\label{tau-rigid finite} Let $A$ be a finite dimensional
  algebra. Then, the following conditions are equivalent:
  \begin{enumerate}
  \item\label{it:tau-rigid finite-a} The algebra $A$ is $\tau$-tilting
    finite.
  \item\label{it:tau-rigid finite-b} The set $\sttilt A$ is finite.
  \item\label{it:tau-rigid finite-c} There exists an upper bound on
    the length of the paths in $\Q(\sttilt A)$.
  \item\label{it:tau-rigid finite-d} There does not exist an infinite
    path starting at $A$ in $\Q(\sttilt A)$.
  \item\label{it:tau-rigid finite-e} There does not exist an infinite
    path ending at $0$ in $\Q(\sttilt A)$.
  \end{enumerate}
\end{proposition}
\begin{proof}

  Firstly, the equivalence between \eqref{it:tau-rigid finite-a} and
  \eqref{it:tau-rigid finite-b} follows from
  \th\ref{char-tau-rigid-finite}.

  Secondly, the implications \eqref{it:tau-rigid
    finite-b}$\Rightarrow$\eqref{it:tau-rigid
    finite-c}$\Rightarrow$\eqref{it:tau-rigid
    finite-d},\eqref{it:tau-rigid finite-e} are obvious. We only show
  that \eqref{it:tau-rigid finite-d}$\Rightarrow$\eqref{it:tau-rigid
    finite-b}; the remaining implication \eqref{it:tau-rigid
    finite-e}$\Rightarrow$\eqref{it:tau-rigid finite-b} can be proven
  in an analogous manner.

  \eqref{it:tau-rigid finite-d}$\Rightarrow$\eqref{it:tau-rigid
    finite-b}
  Let $n:=|A|$. For a support $\tau$-tilting $A$-module $X$, let
  $\ell(X)$ be the supremum of the
  length of the paths starting at $X$ in $\Q(\sttilt A)$.
  
  We claim that $\ell(A)$ is finite.
  Clearly $\ell(X)=\max_Y\ell(Y)+1$ holds, where $Y$ ranges over all
  direct successors of $X$.  Note that $Q(\sttilt A)$ is $n$-regular
  (see \th\ref{QsttiltA}). Thus, $\ell(X)=\infty$ implies $\ell(Y)=\infty$
  for at least one direct successor $Y$ of $X$. Repeating the same
  argument, $\ell(X)=\infty$ implies that there exists an infinite
  path starting at $X$ in $\Q(\sttilt A)$.  Therefore, $\ell(A)$ must be
  finite.

  For any $X\in\sttilt A$, we apply \cite[Thm. 2.35(a)]{adachi_tau-tilting_2014}
  (or \th\ref{main-torsion-classes}(a)) repeatedly to $A\ge X$
  to get a path $A=X_0\to X_1\to\cdots\to X_m=X$ in $\Q(\sttilt A)$ with $m\le\ell(A)$.
  Since $\Q(\sttilt A)$ is $n$-regular, it follows
  immediately that \begin{align*}&|\sttilt A|\leq 1+n+n^2+\cdots+n^{\ell(A)}. \qedhere\end{align*}
  %This concludes the proof.
\end{proof}

To prove \th\ref{all torsion classes are functorially
  finite}, we prepare the following easy observation.
  
  \begin{lemma}
  \th\label{increasing chain}
  Let $A$ be a finite dimensional algebra.
  \begin{enumerate}
  \item If $\T_0\subsetneq\T_1\subsetneq\T_2\subsetneq\cdots$ is a strictly
  increasing chain of torsion classes in $\mod A$, then
  $\T=\bigcup_{i\ge0}\T_i$ is a torsion class which is not functorially finite.
  \item If $\T_0\supsetneq\T_1\supsetneq\T_2\supsetneq\cdots$ is a strictly
  decreasing chain of torsion classes in $\mod A$, then
  $\T=\bigcap_{i\ge0}\T_i$ is a torsion class which is not functorially finite.
  \end{enumerate}
  \end{lemma}

  \begin{proof}
  (a) Clearly $\T$ is a torsion class. Assume that $\T$ is functorially finite.
  Then $\T=\Fac M$ holds for some $M\in\sttilt A$. Since $M\in\T$,
  there exists $i\ge0$ such that $M\in\T_i$. Thus $\T=\Fac M\subset\T_i$ holds,
  contradicting to $\T_i\subsetneq\T_{i+1}$.

  (b) Let $(\T_i,\F_i)$ be a torsion pair and $\F=\bigcup_{i\ge0}\F_i$. 
  By the same argument as in (a),  $\F$ is a torsion-free class which is
  not functorially finite. Since $\T={}^\perp\F$ clearly holds, $(\T,\F)$ is a torsion pair.
  By \th\ref{characterization-of-ftors}, $\T$ is not functorially finite.
  \end{proof}

We now give the proof of \th\ref{all torsion classes are functorially
  finite}.

\begin{proof}[Proof of \th\ref{all torsion classes are functorially finite}]
  Firstly, note that the equivalence between \eqref{it:all tc are
    ff-b} and \eqref{it:all tc are ff-c} follows immediately from
  \th\ref{characterization-of-ftors}.

  \eqref{it:all tc are ff-a}$\Rightarrow$\eqref{it:all tc are ff-b}
  Suppose that there exists a torsion class $\T$ in $\mod A$ which is
  not functorially finite.  Applying \th\ref{main-torsion-classes}
  repeatedly, we obtain an infinite sequence
  \[
    \mod A\supsetneq\Fac N_1 \supsetneq\Fac
    N_2\supsetneq\cdots\supset\T.
  \]
  For all $i$ we have $\Fac N_i\supsetneq\T$ since $\Fac N_i$ is
  functorially finite.  This contradicts the fact that $\sttilt A$ is
  a finite set. Therefore all torsion classes in $\mod A$ are
  functorially finite.

  \eqref{it:all tc are ff-b}$\Rightarrow$\eqref{it:all tc are ff-a}
  Suppose that $A$ is not a $\tau$-tilting finite algebra.  Then, by
  \th\ref{tau-rigid finite}, there exists an infinite path
  $0=M_0<M_1<M_2<\cdots$ ending at $0$  in $\Q(\sttilt A)$. 
  By \th\ref{increasing chain},
  $\bigcup_{i\ge0}\Fac M_i$ is a torsion class which
  is not functorially finite,
  a contradiction. This finishes the proof of the theorem.
\end{proof}

\section{Brick--$\tau$-rigid correspondence}
\label{sec:brick-tau-rigid_correspondence}

Let $A$ be a finite dimensional algebra. We denote the set of
isomorphism classes of indecomposable $\tau$-rigid $A$-modules by
$\itrigid A$. Recall that $X\in\mod A$ is called a \emph{brick} if
$\End_A(X)$ is a division algebra. We denote the set of isomorphism
classes of bricks in $\mod A$ by $\brick A$. We denote the set of
isomorphism classes of bricks $S$ of $A$ such that
the smallest torsion class $\TTT(S)$ containing $S$ is
functorially finite by $\fbrick A$.  The main result of this section
is the following bijective correspondence.

\begin{theorem}
  \th\label{brick-rigid} Let $A$ be a finite dimensional algebra.
  Then there exists a bijection
  \[\itrigid A\to\fbrick A\]
  given by $X\mapsto X/\rad_BX$ for $B:=\End_A(X)$.
\end{theorem}

As an application of \th\ref{brick-rigid}, we prove the following
characterization of $\tau$-tilting finite algebras, which
complements~\th\ref{all torsion classes are functorially finite}.

\begin{theorem}
  \th\label{brick-tau-rigid-wide-finite} Let $A$ be a finite
  dimensional algebra. Then, the following conditions are equivalent.
  \begin{enumerate}
  \item\label{it:b-ttilt-finite} The algebra $A$ is $\tau$-tilting
    finite.
  \item\label{it:b-brick} The set $\brick A$ is finite.
  \item\label{it:b-fbrick} The set $\fbrick A$ is finite.
  \end{enumerate}
  Moreover, in this case there exists a bijection
  \[
    \itrigid A\to\brick A
  \]
  given by $X\mapsto X/\rad_B X$ where $B=\End_A(X)$.
\end{theorem}

For readability purposes, we divide the proof of \th\ref{brick-rigid}
into a couple of lemmas.

\begin{lemma}\th\label{well-defined}
  Let $X$ be an indecomposable $\tau$-rigid $A$-module and
  $B:=\End_A(X)$. Then $S:=X/\rad_BX$ is a brick satisfying
  $\TTT(S)=\Fac X$.
\end{lemma}
\begin{proof}
  Note that $\rad_BX$ is an $A$-module which belongs to $\Fac X$ since
  $\rad_BX=\sum_{f\in\rad B}f(X)$.

  (i) We show that $S$ is a brick. It suffices to show that any
  non-zero morphism $f:S\to S$ is an isomorphism.  Let
  $0\to \rad_BX\xrightarrow{i} X\xrightarrow{p} S\to0$ be an exact
  sequence.  Since $\rad_BX\in\Fac X$ implies $\Ext^1_A(X,\rad_BX)=0$,
  we have the following commutative diagram.
  \[
    \begin{tikzcd}
      0\rar&\rad_B X\rar{i}\dar{h}&X\rar{p}\dar{g}&S\dar{f}\rar&0\\
      0\rar&\rad_B X\rar{i}&X\rar{p}&S\rar&0
    \end{tikzcd}
  \]
  If $g$ is not an isomorphism, then $g(X)\subset\rad_BX$ holds. Thus
  $f=0$, a contradiction.  Thus $g$ is an isomorphism, and hence $h$
  and $f$ are also isomorphisms.

  (ii) We show $\TTT(S)=\Fac X$. It suffices to show that $X$ belongs
  to $\TTT(S)$.  Take $f_1,\ldots,f_n\in B$ satisfying
  $\rad B=\sum_{i=1}^nf_iB$.  Then
  $f=(f_1,\ldots,f_n):X^{\oplus n}\to X$ satisfies
  $f((\rad^i_BX)^{\oplus n})=\rad^{i+1}_BX$.  Therefore
  $\rad^i_BX/\rad^{i+1}_BX$ belongs to $\Fac(\rad^{i-1}_BX/\rad^i_BX)$
  for any $i>0$.  Thus it belongs to $\Fac S$ inductively, and hence
  $X$ belongs to $\Filt(\Fac S)=\TTT(S)$.
\end{proof}

\begin{lemma}
  \th\label{simpleness} Let $S$ be a brick of $A$. Then, the following
  statements hold.
  \begin{enumerate}
  \item\label{a:lemma66} Every morphism $f:X\to S$ in $\TTT(S)$ is
    either zero or surjective.
  \item\label{b:lemma66} If $T$ is another brick of $A$ satisfying
    $\TTT(S)=\TTT(T)$, then $S\simeq T$.
  \end{enumerate}
\end{lemma}
\begin{proof}
  \eqref{a:lemma66} We show that $f\neq0$ implies that $f$ is
  surjective. Since $X\in\TTT(S)=\Filt(\Fac S)$, there exists a
  filtration $0=X_0\subset X_1\subset\cdots\subset X_t=X$ satisfying
  $X_{i+1}/X_i\in\Fac S$ for any $i$.  We can assume $f(X_1)\neq0$ by
  taking a maximal number $i$ satisfying $f(X_i)=0$ and replacing $X$
  by $X/X_i$.  Since $X_1\in\Fac S$, there exists an epimorphism
  $g:S^{\oplus n}\to X_1$.  Since $fg:S^{\oplus n}\to S$ is non-zero
  and $S$ is a brick, $fg$ must be a split epimorphism.  Thus $f$ is
  surjective.

  \eqref{b:lemma66} Since $S$ belongs to $\TTT(T)$, there exists a
  non-zero morphism $f:T\to S$. This is surjective by (a), and
  therefore $\dim_kT\geq\dim_kS$.  The same argument shows the
  opposite inequality, and therefore $f$ is an isomorphism.
\end{proof}

We are ready to prove Theorem \ref{brick-rigid}.

\begin{proof}[Proof of Theorem \ref{brick-rigid}]
  By Lemma \ref{well-defined}, we have a map $\itrigid A\to\fbrick A$
  given by $X\mapsto X/\rad_BX$ for $B:=\End_A(X)$.

  (i) We show that this is injective.  Let $X$ and $Y$ be
  indecomposable $\tau$-rigid $A$-modules and $S$ and $T$ be the
  corresponding bricks.  If $S\simeq T$, then
  $\Fac X=\TTT(S)=\TTT(T)=\Fac Y$. This implies $X\simeq Y$.  In fact,
  otherwise we get a contradiction by applying
  \th\ref{general-observation} to $M:=X\oplus Y$.

  (ii) We show that this is surjective.  Let $S$ be a brick of $A$
  such that $\TTT(S)$ is functorially finite.  Then there exists a
  support $\tau$-tilting $A$-module $M$ satisfying $\Fac M=\TTT(S)$.
  Since $S\in\Fac M$, there exists an indecomposable direct summand
  $X$ of $M$ and a non-zero morphism $f:X\to S$. By Lemma
  \ref{simpleness}, $f$ must be surjective. Thus
  $\TTT(S)\subset\Fac X\subset\Fac M=\TTT(S)$ holds, and hence
  $\TTT(S)=\Fac X$.  Now let $B:=\End_A(X)$ and $T:=X/\rad_BX$. Then
  \[
    \TTT(T)=\Fac X=\TTT(S)
  \]
  holds by \th\ref{well-defined}. By
  \th\ref{simpleness}\eqref{b:lemma66}, we have $S\simeq
  T=X/\rad_BX$. Thus the assertion follows.
\end{proof}

Now we give a proof of \th\ref{brick-tau-rigid-wide-finite}.

\begin{proof}[Proof of \th\ref{brick-tau-rigid-wide-finite}]
  \eqref{it:b-ttilt-finite}$\Rightarrow$\eqref{it:b-brick} Since the
  map $\TTT:\brick A\to\tors A$ is injective by
  \th\ref{simpleness}\eqref{b:lemma66} and the set $\tors A$ is finite
  by \th\ref{all torsion classes are functorially finite}, the set
  $\brick A$ is also finite.

  \eqref{it:b-brick}$\Rightarrow$\eqref{it:b-fbrick} This implication
  is obvious.

  \eqref{it:b-fbrick}$\Rightarrow$\eqref{it:b-ttilt-finite} This
  implication follows directly from \th\ref{brick-rigid}.

  The latter assertion follows from Theorem \ref{brick-rigid} since
  $\tau$-tilting finiteness implies $\tors A=\ftors A$ and hence
  $\brick A=\fbrick A$.
\end{proof}

\section{The simplicial complex associated to $\tau$-tilting modules}
\label{sec:simplicial-complex}

Let $\T$ be a $K$-linear, $\Hom$-finite, Krull-Schmidt triangulated
category $\T$ with a basic silting object $S$ and set $A:=\End_\T(S)$.
In this section we construct a simplicial complex whose maximal
simplices are in bijection with $\twosilt{S}{\T}$ and study its
combinatorial properties.

Let us recall some basic terminology.  Let $\Delta^0$ be a set.  An
\emph{abstract simplicial complex on $\Delta^0$} is a collection
$\Delta$ of finite subsets of $\Delta^0$ closed under taking subsets.
A \emph{simplex of dimension $d$}, or a \emph{$d$-simplex} for short,
is an element of $\Delta$ of the form $\set{v_0,\dots,v_d}$.  We
denote the subset of $\Delta$ consisting of all $d$-simplices by
$\Delta^d$.  For simplicity, we refer to $\Delta^0$ as the set of
\emph{vertices} of $\Delta$.  The \emph{dimension} of $\Delta$ is the
maximum of the dimensions of all its simplices.  We denote the set of
maximal simplices of $\Delta$ by $\Delta^{\max}$.

\begin{definition}
  \th\label{def:DeltaA}
  \begin{enumerate}
  \item We define a simplicial complex $\Delta=\Delta(\T,S)$ on the
    following set:
    \[
      \Delta^0:=\setP{[M]}{M\in\twopresilt{S}{\T}\text{ is
          indecomposable}}.
    \]
    For a subset $\alpha=\set{[M_1],\dots,[M_t]}$ of $\Delta^0$ we
    define
    \[
      M(\alpha):=M_1\oplus\cdots\oplus M_t.
    \]
    Then we declare $\alpha$ to be a simplex of $\Delta$ if
    $M(\alpha)\in\twopresilt{S}{\T}$.
  \item Given a basic two-term presilting complex for $A$
    $M=M_1\oplus\cdots\oplus M_t$ we define
    \[
      \sigma(M):=\set{[M_1],\dots,[M_t]}\in\Delta.
    \]
  \end{enumerate}
\end{definition}

We need to recall more terminology.  A simplicial complex $\Delta$ of
dimension $n$ is \emph{pure} if every maximal simplex of $\Delta$ has
dimension $n$ and every simplex is contained in a maximal simplex. Let
$\Delta$ be a pure simplicial complex of dimension $n$. The boundary
of $\Delta$ consists of the $(n-1)$-simplices which are contained in
exactly one maximal simplex.  We say that $\Delta$ is
\emph{non-branching} if every $(n-1)$-simplex of $\Delta$ is contained
in at most two maximal simplices.

The following result is a reinterpretation of the results of
\cite{adachi_tau-tilting_2014} in terms of the simplicial complex
$\Delta(\T,S)$.

\begin{theorem}
  \th\label{properties-of-Delta} \cite{adachi_tau-tilting_2014} Let
  $n:=|S|$ and $\Delta=\Delta(\T,S)$.  The following statements hold:
  \begin{enumerate}
  \item\label{it:Delta-max-sttiltA} The map $\alpha\mapsto M(\alpha)$
    induces a bijection
    \[
      \begin{tikzcd}
        \Delta^{\max}\rar&\twosilt{S}{\T}
      \end{tikzcd}
    \]
    with inverse $M\mapsto \sigma(M)$.
  \item\label{it:Delta-pure} The simplicial complex $\Delta$ is pure
    of dimension $n-1$.
  \item\label{it:Delta-non-branching} The simplicial complex $\Delta$
    is non-branching and has empty boundary.
  \end{enumerate}
\end{theorem}
\begin{proof}
  \eqref{it:Delta-max-sttiltA} This statement follows immediately from
  \th\ref{tau-tilting-theory}\eqref{it:maximal-tau-rigid} and
  \th\ref{silting-object-gives-basis-of-K0}\eqref{it:bijection-sttilt-twosilt}.
  
  \eqref{it:Delta-pure} By part \eqref{it:Delta-max-sttiltA} we have
  that all maximal simplices of $\Delta$ have dimension $n-1$. The
  fact that every simplex of $\Delta$ is contained in a maximal
  simplex is just a restatement of Bongartz completion for
  $\tau$-rigid $A$-modules
  (\th\ref{tau-tilting-theory}\eqref{it:Bongartz-completion})
  interpreted in $\twosilt{S}{\T}$ via the bijection given in
  \th\ref{silting-object-gives-basis-of-K0}\eqref{it:bijection-sttilt-twosilt}.

  \eqref{it:Delta-non-branching} Every almost complete presilting
  object in $\twopresilt{S}{A}$ is the direct summand of exactly two
  basic silting objects in $\twosilt{S}{\T}$, see
  \th\ref{tau-tilting-theory}\eqref{it:mutation-in-sttiltA} (and again
  use the bijection given in
  \th\ref{silting-object-gives-basis-of-K0}\eqref{it:bijection-sttilt-twosilt}). Therefore
  $\Delta$ is non-branching and has empty boundary.
\end{proof}

Let $\Delta$ be a simplicial complex. We define the \emph{union} and
the \emph{intersection} of a collection of simplices
$\setP{\alpha_i\in\Delta}{i\in I}$ to be the simplicial complexes
\[
  \bigvee_{i\in I}\alpha_i:=\setP{\beta\in\Delta}{\exists i\in I\text{
      such that }\beta\subseteq\alpha_i}
\]
and
\[
  \bigwedge_{i\in I}\alpha_i:=\setP{\beta\in\Delta}{\forall i\in
    I\text{ we have }\beta\subseteq\alpha_i}.
\]
respectively.

\begin{definition}
  Let $\Delta$ be a simplicial complex which is pure of dimension
  $n$. We say that $\Delta$ is \emph{shellable} if there exist a well
  order $\preceq$ on $\Delta^{\max}$, called a \emph{shelling}, such
  that for each $\alpha\in\Delta^{\max}$ the simplicial complex
  \[
    \Delta(\alpha):=\alpha\wedge\left(\bigvee_{\beta\prec\alpha}\beta\right)
  \]
  is pure of dimension $n-1$.
\end{definition}

The following result is analogous to the main result of
\cite{riedtmann_simplicial_1991}.

\begin{theorem}
  \th\label{shellability-sphere} Let $A:=\End_\T(S)$ and suppose that
  $A$ is a $\tau$-tilting finite algebra with $n+1$ simple
  modules. Set $\Delta:=\Delta(\T,S)$.  Then, the following statements
  hold:
  \begin{enumerate}
  \item\label{it:-trf-Delta-shellable} If $n\geq1$, then the
    simplicial complex $\Delta$ is shellable.
  \item\label{it:trf-Delta-sphere} The geometric realization of
    $\Delta$ is homeomorphic to an $n$-dimensional sphere.
  \item If $n\geq2$, then $\Delta$ is simply connected.
  \end{enumerate}
\end{theorem}

We reformulate a special case of results of Danaraj and Klee \cite{danaraj_shellings_1974} in our context (see also \cite{unger_shellability_1999}):
 \begin{theorem}\th\label{danaraj} after \cite[Props. 1.1 and 1.2]{danaraj_shellings_1974}
  A geometric realization of a finite non-branching shellable simplicial complex that is pure of dimension $n$ is homeomorphic to an $n$-dimensional ball or to an $n$-dimensional sphere.  
 \end{theorem}

\begin{proof}[Proof of \th\ref{shellability-sphere}]
  \eqref{it:-trf-Delta-shellable} Let $n\geq 1$. Choose an ordering
  $\alpha_0\prec\alpha_1\prec\dots\prec\alpha_k$ of $\Delta^{\max}$
  such that $M(\alpha_i)< M(\alpha_j)$ implies
  $\alpha_i\prec \alpha_j$.  We need to show that for each
  $0\leq j\leq k-1$ the simplicial complex $\Delta(\alpha_j)$ is pure
  of dimension $n-1$.
  
  The case $j=0$ is obvious, so let $1\leq j\leq k-1$. First, observe
  that every simplex of $\Delta(\alpha_j)$ has codimension at least 1
  in $\Delta$ for it must be properly contained in $\alpha_j$.  Hence,
  to show that $\Delta(\alpha_j)$ is pure of dimension $n-1$ it is
  sufficient to show that every simplex of $\Delta(\alpha_j)$ is
  contained in a simplex of $\Delta(\alpha_j)$ of codimension 1 in
  $\Delta$.  Indeed, let $\beta\in\Delta(\alpha_j)$ and
  $M(\alpha_\ell)$ be the smallest two-term presilting complex having
  $M(\beta)$ as a direct summand (see
  \cite[Prop. 2.9]{adachi_tau-tilting_2014} and use the bijection
  given in
  \th\ref{silting-object-gives-basis-of-K0}\eqref{it:bijection-sttilt-twosilt}). We
  have $M(\alpha_\ell)\neq M(\alpha_j)$ since $\beta\subset \alpha_k$
  for some $k<j$. Hence $M(\alpha_\ell)<M(\alpha_j)$ and, by the
  choice of the ordering, we have $\ell< j$.  By
  \th\ref{main-torsion-classes} (and using the bijection given in
  \th\ref{silting-object-gives-basis-of-K0}\eqref{it:bijection-sttilt-twosilt})
  there exists $M\in\twopresilt{S}{\T}$ obtained by mutation from
  $M(\alpha_j)$ and such that $M(\alpha_\ell)\leq M< M(\alpha_j)$.
  Then, because of the choice of the ordering of $\Delta^{\max}$,
  there exist $m<j$ such that $\sigma(M)=\alpha_m$.  Hence, the
  codimension 1 simplex $\alpha_m\wedge\alpha_j$ is contained in
  $\Delta(\alpha_j)$ and we have $\beta\in(\alpha_m\wedge\alpha_j)$.
  The claim follows.  This shows that $\alpha_1,\dots,\alpha_k$ is a
  shelling of $\Delta$.

  \eqref{it:trf-Delta-sphere} If $n=0$, then $\sttilt A=\set{0,A}$,
  hence
  \th\ref{silting-object-gives-basis-of-K0}\eqref{it:bijection-sttilt-twosilt}
  implies that $\twosilt{S}{\T}=\set{\Sigma S,S}$.  Therefore the
  geometric realization of $\Delta$ consists of two points and hence
  it is homeomorphic to a 0-dimensional sphere.

  Let $n\geq 1$. By \th\ref{properties-of-Delta} and part
  \eqref{it:-trf-Delta-shellable} we have that $\Delta$ is a finite non-branching
  shellable simplicial complex of pure dimension $n$. So, by \th\ref{danaraj}, the geometric realization of $\Delta$
  is homeomorphic to an $n$-dimensional ball or to an $n$-dimensional sphere. As moreover, again by \th\ref{properties-of-Delta}, $\Delta$ has empty boundary, we deduce that the geometric realization of $\Delta$
  is actually homeomorphic to an $n$-dimensional sphere. The last claim then
  follows.
\end{proof}

Let $\Delta$ be a pure simplicial complex. We construct a free
groupoid with set of objects $\Delta^{\max}$ as follows: Let
$\alpha,\beta\in\Delta^{\max}$. If the intersection
$\alpha\wedge\beta$ is a simplex of $\Delta$ of codimension 1, then
there is a pair of mutually inverse isomorphisms
$\varphi_{\alpha,\beta}\colon\alpha\to\beta$ and
$\varphi_{\beta,\alpha}\colon \beta\to\alpha$. which we call
\emph{mutations}. By definition, morphisms in $\Delta^{\max}$ are
given by finite compositions of mutations.

A \emph{cycle} is a sequence of mutations of the form.
\[
  \begin{tikzcd}
    \mu\colon\alpha=\alpha_0\rar{\mu_1}&\alpha_1\rar{\mu_2}&\cdots\rar{\mu_{k-1}}&\alpha_{k-1}\rar{\mu_k}&\alpha_k=\alpha
  \end{tikzcd}
\]
We say that $\mu$ is a \emph{cycle of rank $N$} if
$\bigwedge_{i=0}^k\alpha_i$ is a simplex of codimension $N$ in
$\Delta$. With some abuse of terminology, we say that a cycle is
\emph{generated by cycles of rank $N$} if it belongs to the normal
subgroupoid of $\Delta^{\max}$ generated by all the cycles of rank
$N$.

\begin{theorem}
  \th\label{property-R2} Let $\Delta$ be a shellable simplicial
  complex which is pure of dimension $n$ and such that $\Delta^{\max}$
  is connected as a groupoid. Then, the subgroupoid of $\Delta^{\max}$
  generated by all cycles coincides with the subgroupoid generated by
  all the cycles of rank 2.
\end{theorem}
\begin{proof}
  Fix a shelling $(\Delta^{\max},\preceq)$ of $\Delta^{\max}$ and let
  \[
    \begin{tikzcd}
      \mu\colon\alpha=\alpha_0\rar{\mu_1}&\alpha_1\rar{\mu_2}&\cdots\rar{\mu_{k-1}}&\alpha_{k-1}\rar{\mu_k}&\alpha_k=\alpha
    \end{tikzcd}
  \]
  be a cycle. Without loss of generality, we can assume that two
  subsequent mutations are not inverse of each other.

  First, suppose that $\alpha$ is the least element in
  $(\Delta^{\max},\preceq)$. Let $\alpha_i$ be the greatest element in
  the set $\set{\alpha_0,\alpha_1,\dots,\alpha_k}$ and $t$ be the
  number of times it appears in $\mu$. We shall proceed by induction
  on $(\alpha_i, t)\in(\Delta^{\max},\preceq)\times(\mathbb{N},\leq)$
  ordered lexicographically.

  If $\alpha_0=\alpha_i$, then $k=0$ so the result is trivial in this
  case. Suppose otherwise that $\alpha_0\neq\alpha_i$. Then the
  intersection $\delta:=\alpha_{i-1}\wedge\alpha_i\wedge\alpha_{i+1}$
  is a simplex of codimension 2 in $\Delta$. We claim that there exist
  a simplex $\beta\in\Delta^{\max}$ such that there exist increasing
  sequences of mutations $\xi\colon\beta\to \alpha_{i-1}$ and
  $\zeta\colon\beta\to\alpha_{i+1}$ such that all the maximal
  simplices involved in $\xi$ and $\zeta$ contain $\delta$.

  Indeed, let $\xi'\colon\beta'\to\alpha_{i-1}$ and
  $\zeta'\colon\beta''\to\alpha$ be increasing sequences of mutations
  such that all the maximal simplices involved in $\xi'$ and $\zeta'$
  contain $\delta$. Such sequences exist as we can take
  $\beta'=\alpha_{i-1}$ and $\beta''=\alpha_{i+1}$. Without loss of
  generality, we suppose that $\beta''\prec\beta'$, hence $\delta$ is
  contained in
  \[
    \Delta(\beta'):=\beta'\wedge\left(\bigvee_{\varepsilon\prec\beta'}\varepsilon\right).
  \]
  Since $\Delta(\beta')$ is pure of dimension $n-1$, there exist a
  simplex $\sigma\in\Delta(\beta')^{\max}$, which then has codimension
  1 in $\Delta$, such that $\delta\subset\sigma$. Next, let
  $\gamma\in\Delta^{\max}$ be such that $\sigma\subset\gamma$ and
  $\gamma\prec\beta'$. Since $\sigma\subset(\beta'\wedge\gamma)$,
  there exist a mutation $\nu\colon\gamma\to\beta'$. Also, given that
  $\delta\subset\gamma$, we obtain a new sequence of mutations
  $\nu\xi'$ such that all simplices involved contain
  $\delta$. Finally, as $(\Delta^{\max},\preceq)$ is a well order,
  this process must terminate. This proves the existence of the
  required sequences of mutation $\xi$ and $\zeta$.
  
  Let $\nu=\mu_1\mu_2\cdots\mu_{i-1}$ and
  $\eta=\mu_{i+2}\cdots\mu_{k-1}\mu_k$.  It follows that
  \[
    \mu=\nu\mu_i\mu_{i+1}\eta=(\nu\xi^{-1}\zeta\eta)(\eta^{-1}\zeta^{-1}\xi\mu_i\mu_{i+1}\eta).
  \]
  Note that, by construction, all the simplices in
  $\nu\xi^{-1}\zeta\eta$ are smaller or equal than $\alpha_i$, and
  that $\alpha_i$ appears $t-1$ times in this cycle. Hence, by the
  induction hypothesis, $\nu\xi^{-1}\zeta\eta$ is in the normal
  subgroupoid of $\Delta^{\max}$ generated by the cycles of rank
  2. Next, $\zeta^{-1}\xi\mu_i\mu_{i+1}$ is a cycle of rank 2, for all
  the maximal simplices involved contain the codimension 2 simplex
  $\delta$. Hence $\mu$ is generated by cycles of rank 2. This
  finishes the proof in the case where $\alpha$ is the least element
  in $(\Delta^{\max},\preceq)$.
  
  If $\alpha$ is not the least element in $(\Delta^{\max},\preceq)$,
  let $\alpha'$ be the least element and choose a sequence of
  mutations $\rho\colon\alpha'\to\alpha$. Then, by what we have shown
  above, the sequence of mutations $\rho\mu\rho^{-1}$ is generated by
  cycles of rank 2, so the result holds also in this case.
\end{proof}

\section{$g$-vectors of two-term silting objects}
\label{sec:g-vectors}

Two-term silting complexes in $\KKK{A}$ for a finite dimensional
algebra $A$ have been studied in representation theory by a number of
authors, \eg \cite{hoshino_t-structures_2002,
  derksen_general_2009,adachi_tau-tilting_2014}. They have a numerical
invariant called $g$-vectors, and in this section we study further
combinatorial properties of $g$-vectors in more general triangulated
categories. Our treatment in this section follows closely that of
\cite{dehy_combinatorics_2008}.

\subsection{Basic properties}

Throughout this section, we fix an arbitrary field $K$ and a $K$-linear, $\Hom$-finite,
Krull-Schmidt triangulated category $\T$ with a basic silting object
$S=S_1\oplus\cdots\oplus S_n$, and let $A:=\End_\T(S)$. Recall from
\th\ref{silting-object-gives-basis-of-K0} that we have a bijection
$\twosilt{S}{\T}\to\sttilt A$ and that $K_0(\T)$ has basis
$[S_1],\dots,[S_n]$.

\begin{definition}
  For an object $M\in\T$, we define the following numerical
  invariants:
  \begin{enumerate}
  \item The \emph{$g$-vector of $M$ with respect to $S$} is the
    integer vector $g_S^M={}^\top(g_1,\dots,g_n)$ where
    \[
      [M] = \sum_{i=1}^{n} g_i[S_i]
    \]
    in $K_0(\T)$.
  \item Suppose that $M=M_1\oplus\cdots\oplus M_n$ is a basic silting
    object in $\T$. The \emph{$\GGG$-matrix of $M$ with respect to
      $S$} is the integer matrix
    $\GGG(S,M):=\begin{bmatrix}g_S^{M_1}\mid\cdots\mid
      g_S^{M_n}\end{bmatrix}$.
  \end{enumerate}
\end{definition}

The following observation follows easily from the definitions.

\begin{proposition}
  \th\label{G-identities} Let $L=L_1\oplus\cdots\oplus L_n$ and
  $M=M_1\oplus\cdots\oplus M_n$ be basic silting objects in $\T$. The
  following identities hold:
  \begin{enumerate}
  \item $\GGG(S,L) = \GGG(S,M)\GGG(M,L)$.
  \item\label{it:inverses}
    $\GGG(S,M)\GGG(M,S) = \GGG(M,S) \GGG(S,M) = \mathbf{1}_n$.
  \end{enumerate}
\end{proposition}
\begin{proof}
  The first identity follows since the matrix $\GGG(S,M)$ gives the
  basis change in $K_0(\T)$ from $\set{[S_1],\dots,[S_n]}$ to
  $\set{[M_1],\dots,[M_n]}$.  The second identity now follows.
\end{proof}

We remind the reader that a subset $X$ of $\bigoplus_{i=1}^n\ZZ[S_i]$
is \emph{sign-coherent} if for all $x,y\in X$ with decompositions
\[
  x=\sum_{i=1}^n x_i[S_i]\quad\text{and}\quad y=\sum_{i=1}^n y_i[S_i]
\]
and all $i\in\set{1,\dots,t}$, we have that $x_i$ and $y_i$ are both
non-positive or both non-negative. The following property has been of
interest in (cluster-)tilting theory, see for example \cite[Lemma 1.2]{happel_2005},
\cite[Lemma 2.25]{aihara_silting_2012} and
\cite[Prop. 2.17]{dehy_combinatorics_2008}.

\begin{proposition}
  \th\label{g-vector-determine-M} Let $M\in(\add S)*(\add\Sigma S)$ be
  such that $\T(M,\Sigma M)=0$, and $S'\xto{u} S''\to M\to \Sigma S'$
  be a triangle in $\T$ such that $S',S''\in\add S$ and $u$ is a
  radical morphism. Then $S'$ and $S''$ have no non-zero common direct
  summands.
\end{proposition}

As an immediate consequence we have the following result, which is
given in \cite[Lemma 2.10]{plamondon_generic_2013} and also in
\cite[Thm. 8.1]{cerulli_caldero-chapoton_2015} in a more general
setting.

\begin{theorem}\th\label{sign-coherent}
  \th\label{sign_coherent} The set $\set{g_S^{M_1},\dots,g_S^{M_n}}$
  is sign-coherent.
\end{theorem}

We now study the relation between the partial order on
$\twosilt{S}{\T}$ and their $\GGG$-matrices. 

The following observation for a finite dimensional algebras $A$ over an arbitrary field $K$ justifies the importance of $g$-vectors.  It generalizes some results which are known when $K$ is algebraically closed.

 For $M \in \mod A$, we denote $g^M := g_A^P$ where $P$ is a minimal projective presentation of $M$. For $M, N \in \mod A$, we write $g^M\ge g^N$ if $g^M - g^N$ has non-negative coefficients.

\begin{theorem}
  \th\label{order-implies-epi} Let $A$ be a finite dimensional algebra over an arbitrary field $K$
  and $M$ and $N$ be $\tau$-rigid $A$-modules. Then,
  the following statements hold:
  \begin{enumerate}
  \item\label{it:g-vectors_determine_presilting} If $g^M=g^N$, then
    $M\cong N$.
  \item If $K$ is infinite and $g^M\ge g^N$, then there is a surjective morphism
    $M\to N$.
  \end{enumerate}
\end{theorem}

First we give a proof when $K$ is algebraically closed.

\begin{proof}[Proof of \th\ref{order-implies-epi} when $K$ is algebraically closed]
It suffices to prove (b).
  Let $P_1\xto{a} P_0\to M\to0$ and $Q_1\xto{b} Q_0\to N\to0$ be
  minimal projective presentations for $M$ and $N$. Thus
  $g^{M}=[P_0]-[P_1]$ and $g^{N}=[Q_0]-[Q_1]$.  The group
  $G:=\Aut_A(P_1)\times\Aut_A(P_0)^{\op}$ acts on $\Hom_A(P_1,P_0)$,
  and the orbit $Ga$ is open and dense in $\Hom_A(P_1,P_0)$ with
  respect to Zariski topology by a Voigt-type Lemma
  (\emph{e.g}~\cite[Lemma 2.16]{plamondon_generic_2013}).  Similarly
  $G':=\Aut_A(Q_1)\times\Aut_A(Q_0)^{\op}$ acts on $\Hom_A(Q_1,Q_0)$,
  and the orbit $G'b$ is open and dense in $\Hom_A(Q_1,Q_0)$ with
  respect to Zariski topology.

  By our assumption, there exist a split epimorphism $p:P_0\to Q_0$
  and a split monomorphism $i:P_1\to Q_1$. Now consider maps
  \begin{eqnarray*}
    \pi:\Hom_A(P_1,P_0)\to\Hom_A(P_1,Q_0)&&\pi(f)=fp,\\
    \pi':\Hom_A(Q_1,Q_0)\to\Hom_A(P_1,Q_0)&&\pi'(g)=ig
  \end{eqnarray*}
  which are surjective by our choice of $p$ and $i$.  It is enough to
  show that $\pi(Ga)\cap\pi'(G'b)$ is non-empty since elements
  $g\in G$ and $g'\in G'$ satisfying $\pi(ga)=\pi'(g'b)$ give a
  commutative diagram
  \[
    \begin{tikzcd}
      P_1\rar{ga}\dar{i}&P_0\dar{p}\\
      Q_1\rar{g'b}&Q_0
    \end{tikzcd}
  \]
  and hence we have a surjective morphism
  $M=\Coker(ga)\to\Coker(g'b)=N$.

  Since $Ga$ is a dense subset of $\Hom_A(P_1,P_0)$ and $\pi$ is
  surjective, we have that $\pi(Ga)$ is a dense subset of
  $\Hom_A(P_1,Q_0)$.  On the other hand, by Chevalley's Theorem
  \cite[Ex. II.3.19]{hartshorne_algebraic_1977}, $\pi(Ga)$ is a
  constructible subset of $\Hom_A(P_1,Q_0)$:
  $\pi(Ga)=C_1\cup\cdots\cup C_\ell$ for locally closed subsets
  $C_1,\ldots,C_\ell$.  Since
  \[\Hom_A(P_1,Q_0)=\overline{\pi(Ga)}=\overline{C_1}\cup\cdots\cup\overline{C_\ell}\]
  holds and $\Hom_A(P_1,Q_0)$ is irreducible,
  $\Hom_A(P_1,Q_0)=\overline{C_i}$ holds for some $i$. Then $C_i$ is
  an open subset of $\Hom_A(P_1,Q_0)$, and in particular $\pi(Ga)$
  contains an open dense subset of $\Hom_A(P_1,Q_0)$.

  By the same argument, $\pi'(G'a)$ contains an open dense subset of
  $\Hom_A(P_1,Q_0)$.  Consequently we have
  $\pi(Ga)\cap\pi'(G'b)\neq\emptyset$.
\end{proof}

For a general case, we need the following observations which are of independent interest:

 \begin{proposition}\th\label{field-extension}
  We consider a field extension $K \subset L$, $\overline A := L \otimes_K A$ and $\overline{(-)}:= L \otimes_K-\colon\mod A\to\mod\overline A$ the canonical functor. Then, the following hold for $A$-modules $M$ and $N$:
  \begin{enumerate}
   \item $\Hom_{\overline A}(\overline M, \overline N) = \overline{\Hom}_A(M, N) := L \otimes_K \Hom_A(M, N)$, where $\lambda \otimes f \in L \otimes_K \Hom_A(M, N)$ is identified with $\lambda f \in \Hom_{\overline A}(\overline M, \overline N)$.
   \item $\overline{\Rad}_A(M, N) \subseteq \Rad_{\overline A}(\overline M, \overline N)$.
   \item $g^{\overline M} = \overline g^M$ where $g^{\overline M}$ is the $g$-vector of $\overline M$ as a $\overline A$-module and $\overline g^M$ is the image of $g^M$ by the canonical map $K_0 (\proj A) \to K_0 (\proj \overline A)$ induced by $P \mapsto \overline P$.
   \item $\tau \overline M = \overline{\tau M}$. So $M$ is $\tau$-rigid if and only if $\overline M$ is.   
   \item If $\overline N \in \Fac \overline M$ then $N \in \Fac M$.
  \end{enumerate}
 \end{proposition}

 \begin{proof}
  We fix a minimal projective resolution $\cdots \to P_2 \xto{u_2} P_1 \xto{u_1} P_0 \to M \to 0$.
  
  (a) We get that $\overline P_1 \to \overline P_0 \to \overline M \to 0$ is exact with $\overline P_1$ and $\overline P_0$ projective and therefore $0 \to \Hom_{\overline A} (\overline M, \overline N) \to \Hom_{\overline A} (\overline P_0, \overline N) \to \Hom_{\overline A} (\overline P_1, \overline N)$ is also exact. By a similar argument, we get that $0 \to \overline{\Hom}_{A} (M, N) \to \overline{\Hom}_{A} (P_0, N) \to \overline{\Hom}_{A} (P_1, N)$ is also exact. For $i = 0, 1$, as $P_i$ is projective, it is clear that $\overline{\Hom}_{A} (P_i, N) = \Hom_{\overline A} (\overline P_i, \overline N)$. So $\Hom_{\overline A}(\overline M, \overline N) = \overline{\Hom}_A(M, N)$.

  (b) It is an immediate consequence of the following characterisation of the radical:
    $$\Rad_A(M, N) = \left\{f \in \Hom_A(M, N) \mid \forall g: N \to M, (fg)^{\dim M} = 0\right\}.$$

  (c) As the $u_i$ are radical morphisms, the $\overline u_i$ are also radical so \begin{equation} \label{res-bar} \cdots \to \overline P_2 \xto{\overline u_2} \overline P_1 \xto{\overline u_1} \overline P_0 \to \overline M \to 0\end{equation} is a minimal projective resolution. Hence $g^{\overline M} = \overline g^M$.

  (d) Using $\tau = D \Tr$ where $\Tr$ is the transpose and \eqref{res-bar}, it is clear.
  
  (e) We construct an exact sequence of $K$-vector spaces
   $$M \otimes_K \Hom_A(M, N) \xto{\alpha} N \to C \to 0$$
  where $\alpha$ is the evaluation map. Then applying $L \otimes_K -$, we obtain the exact sequence
   $$\overline M \otimes_L \Hom_{\overline A} (\overline M, \overline N) \xto{\overline \alpha} \overline N \to \overline C \to 0.$$
   As $\overline N \in \Fac \overline M$, we deduce $\overline C = 0$ so $C = 0$. Therefore, $N \in \Fac M$.
 \end{proof}

Now we are ready to complete a proof of \th\ref{order-implies-epi}.

\begin{proof}[Proof of \th\ref{order-implies-epi}]
   Let $L=\overline K$ be an algebraic closure of $K$.
   By \th\ref{field-extension}(d), $\overline M$ and $\overline N$ are $\tau$-rigid $\overline A$-module.
      
   Let us now prove (a). As $g^M = g^N$, we get $g^{\overline M} = g^{\overline N}$ by \th\ref{field-extension}(c). So, using (b) for the $\overline K$-algebra $\overline A$, we get that $\overline M \cong \overline N$. So, using \th\ref{field-extension},
   $$\overline{\Rad}_A(M, N) \subseteq \Rad_{\overline A}(\overline M, \overline N) \subsetneq \Hom_{\overline A}(\overline M, \overline N) = \overline{\Hom}_A(M, N)$$ and, by consequence, $\Rad_A(M, N) \subsetneq \Hom_A(M, N)$. Therefore, $M \cong M' \oplus X$ and $N \cong N' \oplus X$ for some $M', N', X \in \mod A$ with $X \neq 0$, and $g^{M'} = g^{N'}$. By an induction on the number of indecomposable direct summands of $M$, we get that $M' \cong N'$ hence $M \cong N$.

   Finally, we prove (b) when $K$ is an infinite field. As $g^M \ge g^N$, using \th\ref{field-extension}(c), we get $g^{\overline M} \ge g^{\overline N}$ so there is an epimorphism $f: \overline M \to \overline N$. As $\Hom_{\overline A}(\overline M, \overline N) = \overline{\Hom}_A(M, N)$, we can write $f = \sum_{i = 1}^n \lambda_i f_i$ for some $\lambda_i \in \overline K$ and $f_i \in \Hom_A(M, N)$. As $\Hom_{\overline A}(\overline M, \overline N)$ can be embedded in the space of $\dim M \times \dim N$ matrices and $f$ is an epimorphism, there exists a maximal minor $\Delta$ such that $\Delta(f) \neq 0$. We consider the polynomial $R \in K[t_i]_{1 \le i \le n}$ defined by $R(t_1, \dots, t_n) = \Delta(\sum_{i = 1}^n t_i f_i)$. As $R(\lambda_1, \dots, \lambda_n) \neq 0$, we get that $R \neq 0$. Hence, as $K$ is infinite, there exist $\mu_1, \dots, \mu_n \in K$ such that $R(\mu_1, \dots, \mu_n) \neq 0$, hence $\sum_{\i = 1}^n \mu_i f_i \in \Hom_A(M, N)$ is surjective.
  \end{proof}

Let $M=M_1\oplus\cdots\oplus M_t$ be a basic object in $\T$.  We write
$K_0(\T)_{\mathbb{R}}:=K_0(\T)\otimes_{\mathbb{Z}}\mathbb{R}$ and
denote the cone spanned by $\set{[M_1],\dots,[M_t]}$ in $K_0(\T)_\RR$
by $C(M)$.  For example, $C(S)$ coincides with the positive cone
$(\RR_{\geq0})^n$, and $C(\Sigma S)$ coincides with the negative cone
$(\RR_{\leq0})^n$.  The following result, which is analogous to known
results for tilting modules (see \cite{hille_volume_2006} and \cite[Thm. 5.5]{adachi_tau-tilting_2014}) and
cluster-tilting objects (see \cite[Thm. 2.3]{dehy_combinatorics_2008}
and \cite{plamondon_generic_2013}) when $K$ is algebraically closed, follows immediately from
\th\ref{order-implies-epi}.

\begin{corollary}
  \th\label{cones} The following statements hold:
  \begin{enumerate}
  \item The map
    $M\mapsto g_S^M$ induces an injection
    $\twosilt{S}{\T}\to K_0(\T)$.
  \item\label{it:cones-dont-intersect} Let $M,N\in\twosilt{S}{\T}$ be
    basic objects. Then, $C(M)\cap C(N)=C(X)$ where $X$ satisfies
    $\add X =\add M\cap\add N$.
  \end{enumerate}
  In particular, cones determined by different objects in
  $\twosilt{S}{\T}$ intersect only at their boundaries. Therefore,
  $g$-vectors give a natural geometric realization of the simplicial
  complex $\Delta(\T)$.
\end{corollary}

\begin{example}
  \th\label{ex:A3r2-realization} Let $A$ be the algebra given by the
  quiver $3\xla{y} 2\xla{x} 1$ subject to the relation $xy=0$.  The
  geometric realization of $\Delta(A)$ in $K_0(A)_\RR$ using
  $g$-vectors is illustrated in Figure \ref{fig:Delta}.
  \begin{figure}[t]
    \centering \scalebox{0.75}{\tdplotsetmaincoords{100}{155}
\begin{tikzpicture}[scale=4, tdplot_main_coords]
  \draw[color=lightgray,->] (0,0,0) -- (0.7,0,0);
  \draw[color=lightgray,->] (0,0,0) -- (0,0.7,0);
  \draw[color=lightgray,->] (0,0,0) -- (0,0,0.7);
  \draw[color=lightgray,->] (0,0,0) -- (-0.7,0,0);
  \draw[color=lightgray,->] (0,0,0) -- (0,-0.7,0);
  \draw[color=lightgray,->] (0,0,0) -- (0,0,-0.7);
  \node (P1) at (0,0,1) {$0\to P_3$};
  \node (P2) at (-1,0,0) {$0\to P_2$};
  \node (P3) at (0,-1,0) {$0\to P_1$};
  \node (mP1) at ($(0,0,0)-(P1)$) {$P_3\to0$};
  \node (mP2) at ($(0,0,0)-(P2)$) {$P_2\to0$};
  \node (mP3) at ($(0,0,0)-(P3)$) {$P_1\to0$};
  \node (S2) at ($(P2)-(P1)$) {$P_1\to P_2$};
  \node (S3) at ($(P3)-(P2)$) {$P_2\to P_1$};
  % \node (P1) at (0,0,1) {$(P_3;0)=(U;0)$};
  % \node (P2) at (-1,0,0) {$(P_2;0)$};
  % \node (P3) at (0,-1,0) {$(P_1;0)$};
  % \node (mP1) at ($(0,0,0)-(P1)$) {$(0;P_3)$};
  % \node (mP2) at ($(0,0,0)-(P2)$) {$(0;P_2)$};
  % \node (mP3) at ($(0,0,0)-(P3)$) {$(0;P_1)$};
  % \node (S2) at ($(P2)-(P1)$) {$(S_2;0)$};
  % \node (S3) at ($(P3)-(P2)$) {$(S_1;0)$};

  \draw (S3) -- (P3);
  \draw (P2) -- (P3);
  \draw[dashed] (S3) -- (mP2);
  \draw[dashed] (P2) -- (mP3);
  \draw[dashed] (mP2) -- (mP3);
  \draw (S2) -- (P3);
  \draw[dashed] (mP1) -- (mP3);
  \draw[dashed] (P1) -- (mP3);
  \draw[dashed] (S2) -- (mP3);
  \draw[dashed] (P1) -- (mP2);
  \draw[dashed] (mP1) -- (mP2);
  \draw (S3) -- (P1);
  \draw (P1) -- (P2);
  \draw (P1) -- (P3);
  \draw (S2) -- (P2);
  \draw (S2) -- (mP1);
  \draw (P3) -- (mP1);
  \draw (S3) -- (mP1);
\end{tikzpicture}

%%% Local Variables: 
%%% mode: latex
%%% TeX-master: "../master"
%%% End: 
}
    \caption{Geometric realization of $\Delta(A)$, see
      \th\ref{ex:A3r2-realization}.}
    \label{fig:Delta}
  \end{figure}
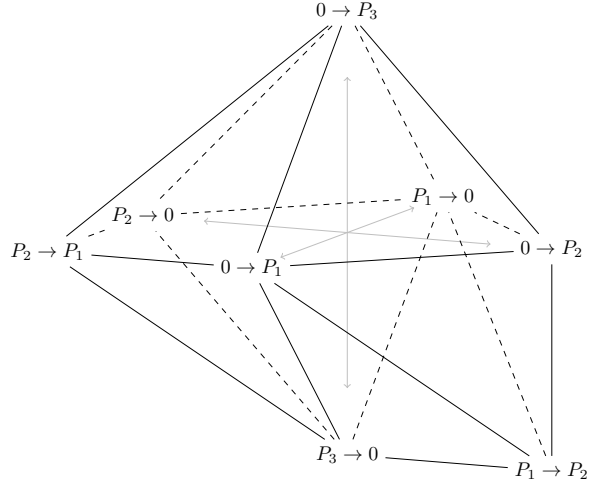
\end{example}

\subsection{The partial order of two-term silting objects in terms of
  $g$-vectors}

In the remainder we fix a Krull-Schmidt triangulated category $\T$
with a silting object $S$. We show that the Hasse quiver
$Q(\twosilt{S} \T)$ of $2$-term silting objects is entirely determined
by the geometry of their $g$-vectors.

We denote the interior of the cone $C(S)$ by $C(S)^{\circ}$.  For a
subset $I$ of $K_0(\T)_{\mathbb{R}}$, we denote the subspace of
$K_0(\T)_{\mathbb{R}}$ spanned by $I$ by $\Span I$ .

\begin{lemma}
  Let $M\in \twopresilt{S}\T$. If $M$ is not silting, then
  $C(S)^{\circ}\cap\Span C(M)=\emptyset$.
\end{lemma}
\begin{proof}
  By Bongartz completion, $M$ is a direct summand of some
  $N\in\twosilt{S}\T$. Take a decomposition
  $N=N_1\oplus\cdots\oplus N_n$ into indecomposable direct summands
  $N_i$, where $M=N_1\oplus\cdots\oplus N_m$ for some $m\le n$.  Since
  $[N_1],\ldots,[N_n]$ is a basis of $K_0(\T)$, we can write
  \[[S_i]=\sum_{j=1}^ng_{ij}[N_j]\] for some integers $g_{ij}$. Then,
  for any $j$ with $1\le j\le n$, we have the following.
  \begin{itemize}
  \item At least one element in $\{g_{1j},\ldots,g_{nj}\}$ is
    non-zero.
  \end{itemize}
  Moreover, by \th\ref{sign_coherent}, we have the following.
  \begin{itemize}
  \item The set $\{g_{1j},\ldots,g_{nj}\}$ is contained in either
    $\mathbb{Z}_{\ge0}$ or $\mathbb{Z}_{\le0}$.
  \end{itemize}
  Any element in $C(S)^{\circ}$ can be written as
  $v=\sum_{i=1}^n\lambda_iS_i$ for some $\lambda_i>0$.  By the above
  two properties, when we express $v$ as a linear combination of
  $[N_1],\ldots,[N_n]$, all the coefficients are non-zero. Therefore,
  if $v$ belongs to $\Span C(M)$, then we have $M=N$. Thus the
  assertion follows.
  % Let $S \cong \bigoplus_{i = 1}^n S_i$ be the decomposition of $S$
  % as a direct sum of indecomposable summands. Using Theorem
  % 2.6\ref{2.6}, the set $\{g_M^{S_1}, \dots, g_M^{S_n}\}$ of
  % $g$-vectors of $S$ with respect to $M$ forms basis of $\RR^n$. As
  % $S \in M[-1] * M$, it is also sign-coherent thanks to Theorem
  % 5.4\ref{5.4}.  Thus, for any $\lambda \in (\RR_{> 0})^n$,
  % $\sum_{i = 1}^n \lambda_i g_M^{S_i}$ has only non-zero
  % entries. Thus, $\lambda$ has no zero coefficient when expressed as
  % a linear combination of the direct summands of $M$ in
  % $K_0(\T)$. Then, it is an immediate consequence of Theorem
  % 5.6\ref{5.6} (b) that $\lambda \notin \Span(C(M) \cap C(N))$.
\end{proof}

Immediately we have the following observation.

\begin{proposition}\label{dichotomy}
  Let $L\in \twopresilt{S}\T$ with $|L|=|S|-1$. Then the hyperplane
  $\Span C(L)$ divides $K_0(\T)_{\mathbb{R}}$ into a half-space
  containing $C(S)$ and a half-space containing $C(\Sigma S)$.
\end{proposition}

Let us call the half-space containing $C(S)$ (respectively,
$C(\Sigma S)$) the \emph{positive} (respectively, \emph{negative})
half-space defined by the hyperplane $\Span C(L)$.

\begin{theorem}\label{orientation}
  Let $M,N\in \twosilt{S} \T$ such that $M$ and $N$ are mutation of
  each other.  Then the following conditions are equivalent, where
  $L\in\T$ is an object satisfying $(\add M)\cap(\add N)=\add L$.
  \begin{itemize}
  \item[(a)] $M>N$ (i.e. there exists an arrow $M\to N$ in the Hasse
    quiver $Q(\twosilt{S} \T)$).
  \item[(b)] $C(M)$ belongs to the positive half-space defined by
    $\Span C(L)$.
  \item[(c)] $C(N)$ belongs to the negative half-space defined by
    $\Span C(L)$.
  \end{itemize}
\end{theorem}

\begin{proof}
  Thanks to Proposition \ref{dichotomy}, it suffices to prove that (a)
  implies (b).  Since $M$ and $N$ are mutation of each other, we can
  write $M=X\oplus L$ and $N=Y\oplus L$ for indecomposable objects $X$
  and $Y$.  Since $M\in\twosilt{S}\T$ implies
  $S\in\twosilt{\Sigma^{-1} M}\T$, there exists a triangle
  \begin{equation}\label{S by M[-1]}
    \Sigma^{-1} M''\to \Sigma^{-1} M' \xrightarrow{f} S \to M''
  \end{equation}
  with $M',M'' \in \add M$. Then $f$ is a right
  $(\add \Sigma^{-1} M)$-approximation of $S$, and we can assume that
  $f$ is right minimal without loss of generality.  In the rest, we
  will prove that $M>N$ implies $X\notin\add M'$ and $X\in\add M''$.

  Since $M>N$, the exchange triangle has the form
  $X\to L' \to Y \to \Sigma X$ with $L'\in\add L$.  Applying
  $\Hom_\T(-, S)$, we have an exact sequence
 $$\Hom_\T(X, \Sigma S) \to \Hom_\T(L', \Sigma S) \to \Hom_\T(Y, \Sigma^2 S) = 0.$$
 Thus $S$ admits a minimal right $(\add \Sigma^{-1}M)$-approximation
 $f:\Sigma^{-1}M' \to S$ with $M' \in \add L$, and therefore the
 triangle \eqref{S by M[-1]} satisfies $X\notin\add M'$ and hence
 $X\in\add M''$.

 By \eqref{S by M[-1]}, we have an equality $[S]=[M'']-[M']$. Thus
 writing $[S]$ as a linear combination of direct summands of $M$ in
 $K_0(\T)$, the coefficient of $[X]$ is positive.  Thus $C(M)$ belongs
 to the positive half-space defined by $\Span C(L)$.
\end{proof}

Recall that a \emph{facet} of a cone $C(M)$ of $M\in\twosilt{S}\T$ is
a face of $C(M)$ in codimension one.

\begin{theorem}\label{recover Hasse}
  Let $\T$ be a triangulated category with $S\in\silt\T$. The Hasse
  quiver $Q(\twosilt{S}\T)$ can be recovered in the following way.
  \begin{itemize}
  \item The vertices are the cones $C(M)$ with $M\in\twosilt{S}\T$.
  \item There is an arrow $M\to N$ if and only if the cone $C(M)$ and
    $C(N)$ are adjacent by a facet $C(L)$, and $C(M)$ belongs to the
    positive half-space defined by $\Span C(L)$.
  \end{itemize}
\end{theorem}
\begin{proof}
  Recall that, for $M,N\in\twosilt{S}\T$, there exists an arrow
  between $M$ and $N$ in the Hasse quiver $Q(\twosilt{S} \T)$ if and
  only if $M$ and $N$ are mutation of each other if and only if the
  cones $C(M)$ and $C(N)$ are adjacent by a facet $C(L)$.  Since the
  orientation of the arrow is given by Theorem \ref{orientation}, we
  have the assertion.
\end{proof}

\begin{corollary}
  \th\label{cones_determine_partial_order} Let $A$ be a $\tau$-tilting
  finite algebra. Then the partial order of support $\tau$-tilting
  $A$-modules is entirely determined by the cones of $g$-vectors.
\end{corollary}

\begin{proof}
  By Theorem \ref{recover Hasse}, the Hasse quiver
  $Q(\twosilt{S}\T)\simeq Q(\sttilt A)$ is determined by the cones of
  $g$-vectors.  Since a partial order on a finite set is recovered
  from its Hasse quiver, we have the assertion.
\end{proof}

We pose the following conjecture.

\begin{conjecture}
  \th\label{the_conjecture} For an arbitrary finite dimensional
  algebra $A$, the partial order of support $\tau$-tilting $A$-modules
  is entirely determined by the cones of $g$-vectors.
\end{conjecture}

The next result gives some information on \th\ref{the_conjecture}.

\begin{proposition}
  \th\label{cone-order} Let $\T$ be a $\Hom$-finite Krull-Schmidt
  triangulated category and $S\in\silt\T$. For $M,N\in\twosilt{S}{\T}$
  we have
  \eqref{it:cone-conj11}$\Rightarrow$\eqref{it:cone-cond1}$\Rightarrow$\eqref{it:MN}
  where:
  \begin{enumerate}
  \item\label{it:cone-conj11} $N\in(\add M)*(\add(\Sigma S))$.
  \item\label{it:cone-cond1} $C(N)\subset C(M)+C(\Sigma S)$.
  \item\label{it:MN} $M\geq N$.
  \end{enumerate}
  Similarly, we have
  \eqref{it:cone-conj22}$\Rightarrow$\eqref{it:cone-cond2}$\Rightarrow$\eqref{it:MNN}
  where:
  \begin{enumerate}
    \setcounter{enumi}{3}
  \item\label{it:cone-conj22} $M\in (\add S)*(\add N)$.
  \item\label{it:cone-cond2} $C(M)\subset C(N)+C(S)$.
  \item\label{it:MNN} $M\geq N$.
  \end{enumerate}
\end{proposition}
\begin{proof}
  The implication
  \eqref{it:cone-conj11}$\Rightarrow$\eqref{it:cone-cond1} is obvious.

  \eqref{it:cone-cond1}$\Rightarrow$\eqref{it:MN} 
   The condition $C(N)\subset C(M)+C(\Sigma S)$ implies $g^{M'}\ge g^N$ for some $M'\in\add M$. Let $A := \End_\T(S)$. We consider the support $\tau$-tilting $A$-modules $X$, $X'$ and $Y$ corresponding to $M$, $M'$ and $N$. Using \th\ref{field-extension} for $L = \overline{K}$, we obtain $g^{\overline X'}\ge g^{\overline Y}$. Thus, by \th\ref{order-implies-epi}, we get $\overline Y \in \Fac \overline X$ so $Y \in \Fac X$ by \th\ref{field-extension}(e). Therefore $M \ge N$.
\end{proof}

Note that the conditions in \th\ref{cone-order} are not equivalent in
general, as shown by the following example.

\begin{example}
  Let $A$ be the preprojective algebra of type $A_3$.  In this case,
  if $M>N$ is a mutation in $\sttilt A$, then
  $C(N)\subset C(M)+C(\Sigma A)$ holds except in the following case
  (up to permutation of 1 and 3).

  Let $M$ and $N$ be support tau-tilting $A$-modules
  $M=\left[\begin{smallmatrix}1\\ 2\end{smallmatrix}\right]\oplus
  \left[\begin{smallmatrix}2\\ 1\end{smallmatrix}\right]$, and let
  $N=\left[\begin{smallmatrix}1\\ 2\end{smallmatrix}\right]\oplus
  S_1$.  Then $M>N$ holds, but the cone $C(N)$ is not contained in
  $C(M)+C(\Sigma A)$.  In fact, $C(N)$ is a non-negative linear
  combination of $[P_1]-[P_3]$, $[P_2]-[P_3]$ and $-[P_3]$.  Therefore
  $[S_1]=[P_1]-[P_2]$ does not belongs to $C(M)+C(\Sigma A)$ since the
  coefficient of $[P_3]$ in $[S_1]$ is 0.
\end{example}

% In view of \th\ref{cones}, it is natural to ask if the partial order
% on $\twosilt{S}{\T}$ can be described purely combinatorially in
% terms of $g$-vectors. We have the following observation.

\subsection{Applications}

In this subsection, we use $g$-vectors to 
classify all two-term silting
complexes for two specific finite dimensional algebras.
Our method is based on the following consequence of \th\ref{cones}:
If we have a set of two term silting objects such that the union of
the corresponding cones is dense in $\RR\otimes_\ZZ K_0(\KKK{A})$,
then it coincides with $\sttilt A$.

Let $A$ be the basic finite dimensional symmetric algebra whose
Gabriel quiver is the quiver
\[
  \begin{tikzcd}
    1\rar[bend left=15]{y}\rar[bend left=50]{x}&2\lar[bend
    left=15]{y}\lar[bend left=50]{x}
  \end{tikzcd}
\]
with relations $xy=yx$ and $x^2=y^2=0$.
Since $A$ is a symmetric algebra, every
silting complex is a tilting complex.

Let $A=X_0\oplus X_1$ so that $g^{X_0}=\gvec{1&0}$ and
$g^{X_1}=\gvec{0&1}$. For each $i\in\ZZ$ there exists an
indecomposable two-term presilting complex $X_i\in\KKK{A}$ such that
\[
  g^{X_i}=\gvec{1-i&i}
\]
and
\[\mu^-_w(A)=
  \begin{cases}
    X_i\oplus X_{i+1}&\mbox{if $w=\cdots2121$ has length $i\ge0$,}\\
    X_{-i}\oplus X_{1-i}&\mbox{if $w=\cdots1212$ has length $i\ge0$.}
  \end{cases}
\]
Thus $Q(\ttwosilt{A})$ has a connected component
\begin{equation}\label{component 1}
  \begin{tikzcd}[row sep=tiny, column sep=small]
	{}&X_1\oplus X_2\rar&X_2\oplus X_3\rar&X_3\oplus X_4\rar&\cdots\\
    A=X_0\oplus X_1\urar\drar\\
    &X_{-1}\oplus X_0\rar&X_{-2}\oplus X_{-1}\rar&X_{-3}\oplus
    X_{-2}\rar&\cdots
  \end{tikzcd}
\end{equation}
Similarly, for each $i\in\ZZ$ there exists an indecomposable two-term
presilting object $Y_i\in\KKK{A}$ such that
\[
  g^{Y_i}=\gvec{i-1&-i}
\]
satisfying $\Sigma A=Y_0\oplus Y_1$ and
\[\mu^+_w(\Sigma A)=
  \begin{cases}
    Y_i\oplus Y_{i+1}&\mbox{if $w=\cdots2121$ has length $i\ge0$,}\\
    Y_{-i}\oplus Y_{1-i}&\mbox{if $w=\cdots1212$ has length $i\ge0$.}
  \end{cases}
\]
Thus $\ttwosilt{A}$ has a connected component
\begin{equation}\label{component 2}
  \begin{tikzcd}[row sep=tiny, column sep=small]
	\cdots\rar&Y_{-3}\oplus Y_{-2}\rar&Y_{-2}\oplus Y_{-1}\rar&Y_{-1}\oplus Y_{0}\drar\\
    &&&&Y_0\oplus Y_1=\Sigma A\\
	\cdots\rar&Y_{3}\oplus Y_{4}\rar&Y_{2}\oplus Y_{3}\rar&Y_{1}\oplus
    Y_{2}\urar
  \end{tikzcd}
\end{equation}
The $g$-vectors of associated to the two-term silting complexes above
are plotted in the following picture of
$K_0(\KKK{A})\cong\ZZ[X_0]\oplus\ZZ[X_1]$:
\begin{equation}
  \label{fig:g-vec-A1}
  \boxinminipage{\tdplotsetmaincoords{180}{135}
\begin{tikzpicture}[scale=0.9, tdplot_main_coords]
  
  \draw[color=lightgray] (-4,4) -- (4,-4);

  \node (X0) at (1,0) {$X_0$};
  \node (X1) at (0,1) {$X_1$};
  \node (X2) at (-1,2) {$X_2$};
  \node (X3) at (-2,3) {$X_3$};
  \node (X4) at (-3,4) {$\cdots$};
  \node (Xm1) at (2,-1) {$X_{-1}$};
  \node (Xm2) at (3,-2) {$X_{-2}$};
  \node (Xm3) at (4,-3) {$\cdots$};
  
  \node (Z) at (0,0) {};

  \node (Y0) at (-1,0) {$Y_0$};
  \node (Y1) at (0,-1) {$Y_1$};
  \node (Y2) at (1,-2) {$Y_2$};
  \node (Y3) at (2,-3) {$Y_3$};
  \node (Y4) at (3,-4) {$\cdots$};
  \node (Ym1) at (-2,1) {$Y_{-1}$};
  \node (Ym2) at (-3,2) {$Y_{-2}$};
  \node (Ym3) at (-4,3) {$\cdots$};

  \path[commutative diagrams/.cd, every arrow]
  (Z.center) edge (X0)
  (Z.center) edge (X1)
  (Z.center) edge (X2)
  (Z.center) edge (X3)
  (Z.center) edge (Xm1)
  (Z.center) edge (Xm2)
  (Z.center) edge (Y0)
  (Z.center) edge (Y1)
  (Z.center) edge (Y2)
  (Z.center) edge (Y3)
  (Z.center) edge (Ym1)
  (Z.center) edge (Ym2);
\end{tikzpicture}
%%% Local Variables: 
%%% mode: latex
%%% TeX-master: "../master"
%%% End: 
}
\end{equation}

\begin{theorem}
  \th\label{qsilt-A1} The quiver $Q(\ttwosilt{A})$ has precisely two
  connected components given by \eqref{component 1} and
  \eqref{component 2}.
\end{theorem}
\begin{proof}

  It readily follows from the above discussion that the union of the
  cones associated to the two-term silting complexes
  $X_i\oplus X_{i+1}$, $Y_i\oplus Y_{i+1}$ is dense in
  $\RR\otimes_\ZZ K_0(\KKK{A})$, see \eqref{fig:g-vec-A1}.
  By \th\ref{cones}, there are no other two-term silting complexes of $A$.
\end{proof}

Now we consider a different algebra, studied originally in
\cite[Ex. 35]{labardini-fragoso_quivers_2009}. Let $B$ be the Jacobian
algebra of the quiver with potential $(Q,W)$ where $Q$ is the quiver
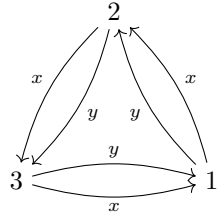
\[
  \begin{tikzpicture}
    \node (1) at (330:1.5) {$1$}; \node (2) at (90:1.5)
    {$2$}; \node (3) at (210:1.5) {$3$};

    \path[commutative diagrams/.cd, every arrow, every label] (1)
    edge[bend left=15] node
    {$y$} (2) (2) edge[bend left=15] node
    {$y$} (3) (3) edge[bend left=15] node {$y$} (1)

    (1) edge[bend right=15] node[swap] {$x$} (2) (2) edge[bend
    right=15] node[swap] {$x$} (3) (3) edge[bend right=15] node[swap]
    {$x$} (1) ;
  \end{tikzpicture}
\]
and $W=x^3+y^3-(xy)^3$. Equivalently, $B=KQ/I$ where $I$ is the ideal
generated by the following relations:
\[
  \begin{cases}
    x^2=yxyxy\\
    y^2=xyxyx\\
    x^2y=xy^2=y^2x=yx^2=0.
  \end{cases}
\]

It is shown in \cite{najera-chavez_c-vectors_2012} that
$Q(\ttwosilt{B})$ has the following connected component which is a
3-regular tree:
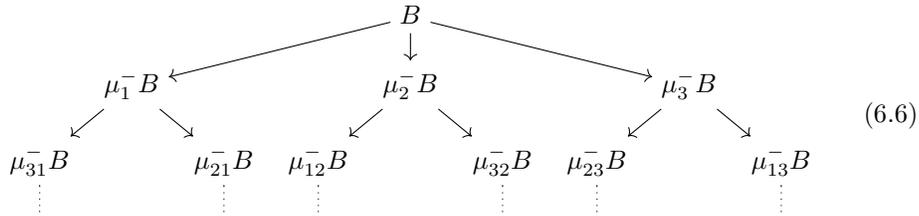
\begin{equation}
  \label{fig:g-vec-B1}
  \boxinminipage{\begin{tikzcd}[row sep=small, column sep=tiny, ampersand replacement=\&]
      {}\&\&\&\&B\dar\ar{llld}\ar{rrrd}\\
      \&\mu_{1}^{-}B\dlar\drar\&\&\&\mu_{2}^{-}B\dlar\drar\&\&\&\mu_{3}^{-}B\dlar\drar\\
      \mu_{31}^{-} B\dar[dotted,path]\&\&\mu_{21}^{-} B\dar[dotted,path]\&\mu_{12}^{-} B\dar[dotted,path]\&\&\mu_{32}^{-} B\dar[dotted,path]\&\mu_{23}^{-} B\dar[dotted,path]\&\&\mu_{13}^{-} B\dar[dotted,path]\\
      \phantom{\tiny{.}}\&\&\phantom{\tiny{.}}\&\phantom{\tiny{.}}\&\&\phantom{\tiny{.}}\&\phantom{\tiny{.}}\&\&\phantom{\tiny{.}}
    \end{tikzcd}
  }
\end{equation}
and that the closure of the union of the cones of two-term silting
complexes in this connected component is
$\setP{(x,y,z)\in\RR^3}{x+y+z\geq0}$.  It follows by symmetry that
$Q(\ttwosilt{B})$ also has the following connected component:

\begin{equation}
  \label{fig:g-vec-B2}
  \boxinminipage{
    \begin{tikzcd}[row sep=small, column sep=tiny, ampersand replacement=\&]
      \phantom{\tiny{.}}\dar[dotted,path]\&\&\phantom{\tiny{.}}\dar[dotted,path]\&\phantom{\tiny{.}}\dar[dotted,path]\&\&\phantom{\tiny{.}}\dar[dotted,path]\&\phantom{\tiny{.}}\dar[dotted,path]\&\&\phantom{\tiny{.}}\dar[dotted,path]\\
      \mu_{31}^{+} \overline{B}\drar\&\&\mu_{21}^{+} \overline{B}\dlar\&\mu_{12}^{+} \overline{B}\drar\&\&\mu_{32}^{+} \overline{B}\dlar\&\mu_{23}^{+} \overline{B}\drar\&\&\mu_{13}^{+} \overline{B}\dlar\\
      \&\mu_{1}^{+}\overline{B}\ar{rrrd}\&\&\&\mu_{2}^{+}\overline{B}\dar\&\&\&\mu_{3}^{+}\overline{B}\ar{llld}\\
      {}\&\&\&\&\overline{B}
    \end{tikzcd}
  }
\end{equation}
where we write $\overline{B}:=\Sigma B$ to save space. In this case,
the closure of the union of the cones of two-term silting complexes in
this connected component is $\setP{(x,y,z)\in\RR^3}{x+y+z\leq0}$.

The following result can be proven analogously to \th\ref{qsilt-A1}.

\begin{theorem}
  The quiver $Q(\ttwosilt{B})$ has precisely two connected components
  given by \eqref{fig:g-vec-B1} and \eqref{fig:g-vec-B2}.
\end{theorem}

\subsection{Connection to cluster-tilting theory}

Let $\C$ be a $K$-linear, $\Hom$-finite, Krull-Schmidt, 2-Calabi-Yau
triangulated category with a basic cluster-tilting object
$T=T_1\oplus\cdots\oplus T_n$. We remind the reader that we have
$\C=\add T*\add\Sigma T$, see for example
\cite[Sec. 2.1]{keller_cluster-tilted_2007}. We denote the set of
isomorphism classes of basic cluster-tilting objects in $\C$ by
$\ctilt\C$.

We note that in general the indecomposable direct summands of $T$ do
not form a basis of the Grothendieck group of $\C$. Thus, we consider
the split Grothendieck group $K_0^\oplus(T)$ of the additive category
$\add T$. That is, $K_0^\oplus(T)$ is the quotient of the free abelian
group generated by the isomorphism classes of objects in $\add T$ by
the subgroup generated by all elements of the form
\[ [T'\oplus T''] - [T'] - [T''].
\]
Thus we have
$K_0^\oplus(T)=\ZZ[T_1]\oplus\ZZ[T_2]\oplus\cdots\oplus\ZZ[T_n]$.

Let $M\in\C$. There exists a triangle
\begin{equation}
  \label{triangle-M}
  T^1\to T^0\xto{f} M\to \Sigma T^1  
\end{equation}
such that $T^0,T^1\in\add T$ and $f$ is a minimal right
$(\add T)$-approximation. The \emph{index of $M$ with respect to $T$}
is defined by
\[
  \ind_T(M):=[T^0]-[T^1]\in K_0^\oplus(T).
\]
Dually, consider a triangle
$M\xto{g}\Sigma^2({}^0T)\to{}\Sigma^2(^1T)\to \Sigma M$ with
${}^0 T,{}^1 T\in\add T$ and $g$ is a minimal left
$(\add T)$-approximation. The \emph{coindex of $M$ with respect to
  $T$} is defined by
\[
  \coind^T(M):=[{}^0 T]-[{}^1 T]\in K_0^\oplus(T).
\]
Note that we have $\coind^T(M)=-\ind^T(\Sigma^{-1}M)$

Let $M=M_1\oplus\cdots\oplus M_n\in\ctilt\C$. The \emph{$\GGG$-matrix
  of $M$ with respect to $T$} is the integer matrix
$\GGG(T,M):= [g_T^{M_1}\mid\cdots\mid g_T^{M_n}]$ where $g_T^{M_i}$ is
the column vector corresponding to $\ind_T(M_i)$ in the ordered basis
$\set{[T_1],\dots,[T_n]}$.  Similarly, we define the
\emph{$\CCC$-matrix of $M$ with respect to $T$} using coindices, and
denote it by $\CCC(T,M)$.

Let $A:=\End_\C(T)$. Recall from
\cite[Prop. 2(c)]{keller_cluster-tilted_2007} that the functor
$\Hom_\C(T,-)\colon\C\to\mod A$ induces an equivalence of categories
\[
  \begin{tikzcd}
    F\colon\C/[\Sigma T]\rar{\sim}&\mod A
  \end{tikzcd}
\]
where $[\Sigma T]$ is the ideal of $\C$ of morphisms that factor
through $\add\Sigma T$. Moreover, it is shown in
\cite[Thm. 4.1]{adachi_tau-tilting_2014} that $F$ induces a bijection
\begin{equation}
  \label{ctiltC-sttiltA}
  \begin{tikzcd}
    \ctilt\C\rar{F} & \sttilt A
  \end{tikzcd}
\end{equation}
given by $M=(X\oplus \Sigma T')\mapsto (FX,FT')$ where $\Sigma T'$
satisfies $\add M\cap\add\Sigma T=\add \Sigma T'$ and $X$ has no
indecomposable direct summands in $\add\Sigma T$.  Therefore, using
\th\ref{silting-object-gives-basis-of-K0}\eqref{it:bijection-sttilt-twosilt}
we deduce that the map
$M=X\oplus\Sigma T'\mapsto \Sigma FT'\oplus P(FX)$ where $P(FX)$ is a
minimal projective presentation of the $A$-module $FX$ induces a
bijection

\begin{equation}
  \label{ctiltC-twosiltA}
  \begin{tikzcd}[row sep=tiny]
    \ctilt\C\rar{\widetilde{(-)}}&\ttwosilt{A}.
  \end{tikzcd}
\end{equation}

\begin{theorem}
  \th\label{CG-identitites}
  Let $M\in\ctilt\C$. The following identities hold:
  \begin{enumerate}
  \item\label{it:g-matrix} $\GGG(T,M)=\GGG(A,\widetilde{M})$.
  \item\label{it:c-matrix}
    $\CCC(\Sigma^{-2}M,T)=\GGG(\widetilde{M},A)$.
  \end{enumerate}
\end{theorem}
\begin{proof}
  Part \eqref{it:g-matrix} is clear, since by definition a two-term
  complex in $\KKK{A}$ belongs to $\add A*\add(\Sigma A)$.

  \eqref{it:c-matrix} Let
  \begin{equation}
	\label{tri-1}
    T\xto{f}M'\xto{g} M''\xto{h} \Sigma T
  \end{equation}
  be a triangle in $\C$ where $f$ is a minimal left
  $(\add M)$-approximation. It follows from the minimality of $f$ that
  $M'$ has no indecomposable direct summands in $\add\Sigma T$.
  Moreover, since $M$ is rigid, $h$ is a right
  $(\add M)$-approximation and we can write $M''=Y\oplus (\Sigma T')$
  where $\Sigma T'$ satisfies $\add M\cap\add\Sigma T=\add \Sigma T'$
  and $Y$ has no indecomposable direct summands in $\add\Sigma T$. It
  follows from the bijection \eqref{ctiltC-sttiltA} that $(FM,FT')$ is
  a support $\tau$-tilting pair of $A$-modules. In addition, applying
  $F$ to the triangle \eqref{tri-1} yields an exact sequence
  \[
    A\xto{Ff}FM'\xto{Fg}FY\to0
  \]
  where $Ff$ is a minimal left $(\add FM)$-approximation.

  On the other hand, let $\widetilde{M}\in\ttwosilt{A}$ be the silting
  complex corresponding to $M$ via the bijection
  \eqref{ctiltC-twosiltA}. Let
  \begin{equation}
    \label{tri-2}
	A\xto{\alpha}\widetilde{M}'\xto{\beta}\widetilde{M}''\xto{\gamma}\Sigma A
  \end{equation}
  be a triangle in $\KKK{A}$ such that $\alpha$ is a minimal left
  $(\add \widetilde{M})$-approximation.  As above, we may write
  $\widetilde{M''}=\widetilde{Y}\oplus \Sigma P$ where $\Sigma P$
  satisfies $\add \widetilde{M}\cap\add \Sigma A=\add \Sigma P$ and
  $\widetilde{Y}$ has no indecomposable direct summands in
  $\add \Sigma A$.  Taking 0-th cohomomology of the triangle
  \eqref{tri-2} we obtain an exact sequence
  \[
    A\xto{H^0(\alpha)}H^0(\widetilde{M}')\xto{H^0(\beta)}H^0(\widetilde{Y})\to0
  \]
  where $H^0(\alpha)$ is a minimal left $(\add FM)$-approximation
  (note that $FM\cong H^0\widetilde{M}$ by construction). It follows
  that $FM'\cong H^0(\widetilde{M}')$ and
  $FY\cong H^0(\widetilde{Y})$.

  Finally, since $\ker F =\add \Sigma T$ and
  $\ker H^0(-)|_{\add A*\add(\Sigma A)}=\add \Sigma A$ , by comparing
  summands in the triangles \eqref{tri-1} and \eqref{tri-2} we deduce
  that
  $\coind^{\Sigma^{-2}M}(T)=[\Sigma^{-2}M']-[\Sigma^{-2}M'']\in
  K_0^\oplus(M)$ and
  $[A]=[\widetilde{M}']-[\widetilde{M}'']\in K_0(\KKK{A})$ have the
  same coordinates.
\end{proof}

We obtain a new proof of the following result in cluster-tilting
theory, see for example \cite[Thm. 4.1 and
Rmk. 4.2]{nakanishi_periodicities_2011}.

\begin{corollary}
  If $M$ is a cluster-tilting object in $\C$, then
  \[
    \CCC(\Sigma^{-2}M,T) \GGG(T,M)=\GGG(T,M)
    \CCC(\Sigma^{-2}M,T)=\mathbf{1}_n.
  \]
\end{corollary}
\begin{proof}
  By \th\ref{G-identities,CG-identitites} we have
  \[
    \CCC(\Sigma^{-2}M,T) \GGG(T,M)=\GGG(\widetilde{M},A)
    \GGG(A,\widetilde{M})=\mathbf{1}_n.\qedhere
  \]
\end{proof}

%%% Local Variables:
%%% mode: latex
%%% TeX-master: "master-imrn"
%%% End:

%%%%%%%%%%%%%%%%%%%%%%%%%%%%%%%%%%%%%%%%%%%%%%%%%%%%%%%%%%%% 

\bibliographystyle{alpha}
\bibliography{zotero}

\end{document}